\documentclass[12pt]{article}
\usepackage[utf8]{inputenc}

\usepackage{amssymb,enumitem,xcolor}

\usepackage{lmodern,microtype}

\usepackage{mathtools,amsthm,icomma,upgreek,xfrac,amsmath,dsfont}
\mathtoolsset{showonlyrefs,mathic}

\usepackage[colorlinks=true, urlcolor=blue, linkcolor=blue, citecolor=blue]{hyperref}

\allowdisplaybreaks

\newcommand{\defeq}{\mathrel{\mathop:}=}

\newcommand{\E}{\mathbf{E}}
\renewcommand{\P}{\mathbf{P}}
\newcommand{ \p}{\boldsymbol{p}}
\newcommand{\R}{\mathbf{R}}
\newcommand{\N}{\mathbf{N}}
\newcommand{\I}{\mathds{1}}
\newcommand{\rh}{\mathbb{H}}

\newcommand{\ins}{\text{ e.ins.}}
\newcommand{\sep}{\text{ e.sep.}}
\newcommand{\Ins}{edgewise inseparable }
\newcommand{\Sep}{edgewise separable }

\theoremstyle{plain}
\newtheorem{theorem}{Theorem}[section]
\newtheorem{lemma}[theorem]{Lemma}
\newtheorem{proposition}[theorem]{Proposition}
\newtheorem{cor}[theorem]{Corollary}
\theoremstyle{definition}
\newtheorem{definition}[theorem]{Definition}

\newtheorem{example}{Example}[section]
\newtheorem{remark}[theorem]{Remark}

\numberwithin{equation}{section}

\newcommand{\Var}{\text{Var}}

\def\({\left(}
\def\){\right)}
\def\[{\left[}
\def\]{\right]}

\title {Normal approximation for subgraph count\\ in random hypergraphs
\footnotetext{2020 MS Classification: 05C65, 05C80, 	60F05.\\
{\it Key words and phrases}: random hypergraph, subgraph count, asymptotic normality, Kolmogorov distance, Wasserstein distance.
       }}
\author{Wojciech Michalczuk\thanks{ Wroc{\l}aw University of Science and Technology, Ul. Wybrze\.ze Wyspia\'nskiego 27, Wroc{\l}aw, Poland.
e-mail: {\tt w.s.michalczuk@gmail.com}.}
\and 
Miko{\l}aj Nieradko\thanks{ Wroc{\l}aw University of Science and Technology, Ul. Wybrze\.ze Wyspia\'nskiego 27, Wroc{\l}aw, Poland.
e-mail: {\tt 	mikolaj.nieradko@gmail.com}.}
\and 
Grzegorz Serafin\thanks{ Wroc{\l}aw University of Science and Technology, Ul. Wybrze\.ze Wyspia\'nskiego 27, Wroc{\l}aw, Poland.
e-mail: {\tt grzegorz.serafin@pwr.edu.pl}.}}

\begin{document}
\maketitle

\begin{abstract} 
A non-uniform and inhomogeneous random hypergraph model is considered, which is a straightforward extension of the celebrated  binomial random graph model $\mathbb G(n, p)$. We establish necessary and sufficient conditions for small hypergraph count to be asymptotically normal, and complement them with convergence rate in both the Wasserstein and Kolmogorov distances. Next we narrow our  attention to the homogeneous model and relate the obtained results to the fourth moment phenomenon. Additionally, a short proof of necessity of aforementioned conditions  is presented, which seems to be absent in the literature even in the context of the model $\mathbb G(n, p)$. 
\end{abstract}

\section{Introduction}

Let $\rh(n,  \p)$,  $ \p=(p_1, p_2,\ldots, p_{n})$,  be a random hypergraph model on vertices from the set $[n]:=\{1,2,\ldots,n\}$, in which every non-empty edge $e$ exists independently with probability $p_{|e|}\in(0,1)$ depending on its size $|e|$. In particular, loops are included. When restricted to edges of size $2$ only,  we obtain  the celebrated binomial (Gilbert-Erd\H os-R\'enyi) random graph model $\mathbb G(n, p) $. In order to understand better the context of   this article,  let us briefly introduce the model $\mathbb G(n, p) $,  and recall its properties that are relevant from our point of view.

In the  random graph model $\mathbb G(n, p) $ any two  of $n$ vertices are independently connected with probability $p=p(n)$. When $n\rightarrow\infty$ and $p$ depends on $n$ in a suitable manner, one can observe many interesting phenomena, that have been subject of research for last decades (see  \cite{FK, JLR} and references therein). One of the studied quantities is so-called small subgraph count. Namely, for a fixed graph $G$  we denote by $Z_G^n$ the random variable that counts the number of subgraphs of $\mathbb G(n, p)$ that are isomorphic to $G$. It is shown in \cite{Ruc88}  that  the standardized variable $\widetilde Z_G^n=(Z_G^n-\E[Z_G^n])/\sqrt{\Var[Z_G^n]}$  converges in distribution to a r.v. $\mathcal N$ with the standard normal distribution $\mathcal N(0,1)$ if and only  if
\begin{align}\label{eq:ERns}
\min_{F\subseteq G:\,e_F\geq1}\{n^{v_F}p^{e_F}\}\rightarrow\infty\ \ \ \ \text{ and }\ \ \ \  (1-p)n^2\rightarrow\infty,
\end{align}
where $v_F$ and $e_F$ stand for the numbers of vertices and edges of $F$, respectively. Later on, this result has been complemented with the convergence rate of the form \cite{BKR} 
\begin{align}\label{eq:d_1}
d_1(\widetilde Z_G^n, \mathcal N):=\sup_{h\in \mathcal L}\big|\E[h(\widetilde Z_G^n)]-\E[h(\mathcal N)]\big|\leq C\,\big((1-p)\min_{F\subseteq G:\,e_F\geq1}\{n^{v_F}p^{e_F}\}\big)^{-1/2},
\end{align}
where  $C$ is a constant depending on $G$ and $\mathcal L$ is the class of functions $h:\R\rightarrow \R$ satisfying $\|h\|_\infty+\|h'\|_\infty\leq1$. The conjecture that the   same rate holds for more desired Kolmogorov distance (cf. \eqref{eq:d_K}) remained unsolved for the next  three decades, and  eventually it was confirmed in  \cite{PSbej}, see also \cite{R, Z, ER}.

% In this article, we address an analogous problem in the random hypergraph  model $\rh(n, \p)$. For a fixed hypergraph $H$  we denote by $Z_H^n$ the number of isomorphic copies of $H$ in  $\rh(n, \p)$, and let $\widetilde Z_H^n$ be standardisation of $Z_H^n$
% $$\widetilde Z_H^n=\frac{Z_H^n-\E[Z_H^n]}{\sqrt{\Var[Z_H^n]}}.$$
 There exists a vast literature on random hypergraphs, however, mostly uniform (every edge is of the same size) and homogeneous (edge existence probabilities are the same) models are considered, see e.g., \cite{KL, BR, BR2, LZ},  and also \cite{B, C, CDP, DN, GZCN} for some other models. 
   In \cite{dJhyper}, necessary and sufficient conditions for asymptotic normality of $\widetilde Z_H^n$ have been obtained for  $k$-uniform hypergraphs $H$.
%  , which means that  all its hyperedges consist of exactly $k$ vertices.
 Comparing to \eqref{eq:ERns}, the only change is that $n^2$ is replaced with $n^k$.
From {a} technical point of view, deriving convergence rates of $\widetilde Z_H^n$ to the standard normal distribution seems very similar in random graphs and uniform random hypergraphs. This might explain, why there is no separate result of this kind for  the latter ones. As mentioned in the last paragraph, even for random graphs estimates of the Kolmogorov distance are very recent \cite{PSbej}. In inhomogeneous random (hyper)graphs some novel {phenomena} occur, as for instance some of the edge probabilities may be close to zero, while other ones are close to one. A  uniform random hypergraph with strong inhomogeneity has been studied in \cite{CDP} in the context of large deviations. On the other hand, non-uniformity creates another kind of obstacles. For example, when all of the edge probabilities are close to one, the subgraph count might be approximated by a sum of independent random variables,  nonetheless the Berry-Esseen  theorem still does not provide us with bounds recovering conditions for asymptotic normality. The level of generality presented in this article has been   recently addressed 
in  \cite{Dewar}, where  some kind of a counterpart  of a threshold for the property $\{Z_H^n>0\}$ is derived. Also, a bound on the Wasserstein distance between $\widetilde Z_H^n$ and  $\mathcal N$ in a similar model is established in \cite{V}.

Following the history of research on  $Z_G^n$ in the binomial model $\mathbb G(n, p)$, we devote this article to examining  the asymptotic normality of  $\widetilde Z_H^n$, where $Z_H^n$ stands for the  variable counting copies of a hypergraph $H$ in $\rh(n,  \p )$.
Precisely,  we derive  necessary and sufficient conditions for  $\widetilde Z_H^n$ to {be} asymptotically normal,  and complement them with  convergence rates in both: Wasserstein and Kolmogorov distances. The obtained results extend the ones known for   $\mathbb G(n, p)$. Next, we narrow our attention to the homogeneous case, where  $  \p =(p ,p,\ldots , p)$ for some $p\in(0,1)$, and  express the aforementioned results by means of the excess kurtosis $\E[(\widetilde Z^n_H)^4 ]-3$. This refers  to the so-called  {\it fourth  moment phenomenon}, which has been initially studied by Peter de Jong \cite{dJ87, dj90} for general $U$-statistics  and recently gained a lot of attention in the field of stochastic analysis, where its quantitative versions in various settings are studied \cite{NP, DP, DK, D24, L}. Additionally, we provide a  very simple proof of the necessary conditions for the asymptotic normality of $\widetilde Z_H^n$, which seems to be new even for the  model $\mathbb G(n, p)$.

The complex structure of the model $\rh(n, \p)$ makes it resistant to many methods applied to the $\mathbb G(n, p)$ model, including the classical ones from \cite{BKR, JLR, dJhyper, Ruc88}.   In order to obtain the convergence rates we employ the tools developed in \cite{PSejp2, PSbej, PSstoch}, as they seem more flexible. Nevertheless, direct application does not lead to satisfactory results, therefore much more  subtle strategy is implemented.  The starting point of most of the proofs is  the Hoeffding decomposition of $\widetilde Z_H^n$, which is not only needed for the aforementioned approach, but also allows us to determine the core terms responsible for the asymptotic behaviour. To obtain the necessary condition the Feller-L\'evy condition is verified, and then Lindeberg condition is presented in the simplest possible form. Let us note here that stronger Lyapunov condition fails in the case of inhomogeneous model.  
 Finally, the derivation of the  results related to the fourth moment phenomenon relies mainly on  combinatorial arguments.

The paper is organized as follows. In Section 2 we describe the model $\rh(n, \p)$ in detail, establish and discuss the Hoeffding decomposition of the variable $\widetilde Z_H^n$, and also gather essential notation used in the article. Section 3 is devoted to the main results of the paper,  while the remaining sections contain  their proofs (for $H$ without isolated vertices, see Remark \ref{rem:isolated}). More precisely,  Section 4 deals with convergence rates, which are then estimated by the excess kurtosis in Section 5. In Section 6  the necessary and sufficient  conditions for the asymptotic normality of $\widetilde Z_H^n$ are derived. Finally, in Appendix we provide some general properties of the Kolmogorov and Wasserstein distances.

\section{Preliminaries}
\subsection{Basic properties of the model $\rh(n,  \p )$}
As mentioned in Introduction, the random hypergraph model $\rh(n,  \p )$ is built on $n$ vertices, which we take from the set $[n]=\{1, 2, 3,\ldots , n\}$. We impose no restrictions on sizes of edges (or hyperedges), thus  any of $2^n-1$ edges from the set $2^{[n]}\backslash \{\emptyset\}$ may occur. The complete hypergraph $([n], 2^{[n]}\backslash \{\emptyset\})$ will be denoted by $K_n^*$. 
The vector $  \p =  \p (n)$ is the vector of edge probabilities, i.e. every edge   $e$ of size $|e|$ exists, independently of the other edges, with the probability $p_{|e|}\in(0,1)$.
 
 Our interest is focused on the random variable counting isomorphic copies of a given hypergraph $H$ in the random hypergraph $\rh(n,   \p )$. Precisely, for a fixed $H$ we define 
 \begin{equation}\label{form:copy1}
		Z_H^n=\sum_{\substack{H'\subseteq K^*_n \\ H' \simeq H}} \I _{\left\{ H' \subseteq \rh(n, \boldsymbol{p}) \right\}}.
	\end{equation}
Here, the relation $\simeq$ denotes standard isomorphism between two hypergraphs.   Furthermore, $\subseteq$ stands for strong (and not necessarily induced) inclusion of hypergraphs, i.e. for two hypergraphs $F$ and $H$ we write $F\subseteq H$ whenever $V(F)\subseteq V(H)$ and $E(F)\subseteq E(H)$, where $V(H)$ and $E(H)$  are the number of vertices and edges, respectively, of the hypergraph $H$.  Let us note that in the literature one often writes $\left\{ H \subseteq \rh(n, \boldsymbol{p}) \right\}$ in the sense of $\left\{Z_H^n\geq1\right\}$.  For clarity, we use only the latter notation of this property.

By $\widetilde Z_H^n$ we denote the standardization of the  random variable $Z_H^n$, i.e.
$$\widetilde Z_H^n=\frac{Z_H^n-\E[Z_H^n]}{\sqrt{\Var{Z_H^n}}}$$ 
Expected value of $Z_H^n$ is easy to calculate and is given by $\E[Z_H^n]=N_H^n  \prod _{e\in E(H)}p_{|e|}$, where
$N_H^n:=|\{H'\subseteq K^*_n:H' \simeq H\}|$  is the total number of copies of $H$ in the complete hypergraph $K_n^*$. In the case of the variance, the formula is  more complicated. Nevertheless, from our point of view estimates are sufficient and they are given by (see \cite[proof of Theorem 3.2]{Dewar}, and also formula \eqref{eq:VarZ} below for the variance) 
\begin{align}\label{eq:Eestgen}
\E[Z_H^n]&\approx n^{v_H}P_H,\\\label{eq:Vestgen}
\Var[Z_H^n]&\approx P_H^2\max_{F\subseteq H, e_F \geq 1} n^{2v_H-v_F} \frac{1 - P_F }{P_F}\approx P_H^2\max_{F\subseteq H, e_F \geq 1} n^{2v_H-v_F} \frac{Q_F}{P_F},
\end{align}
where $v_H=|V(H)|$ and 
\begin{align*}
P_H= \prod _{e\in E(H)}p_{|e|}, \hspace{30pt }Q_H= \prod _{e\in E(H)}(1-p_{|e|}).
\end{align*}
The notation $\approx $ is understood as follows:  for two sequences of non-negative functions $f_n$ and $g_n$ we write $f_n\approx g_n$ whenever
$$C_1 f_n\leq g_n\leq C_2f_2$$
holds on the indicated domain  for almost all $n\in\N$. Here,  $C_1, C_2>0$ are constant possibly dependent on the fixed hypergraph $H$. If they also  depend on some other parameter $a$, we indicate it by $\stackrel{a}{\approx}$. Additionally, if we are interested only in one of inequalities above, we write $f_n\lesssim g_n$ or $g_n\lesssim f_n$.

\subsection{Hoeffding decomposition }
The Hoeffding decomposition is a very useful tool originally introduced in \cite{H} for symmetric $U$-statistics. Below, we present it in the general setup. It has been already exploited in the context of subgraph count in random graphs in  \cite{Bloznelis, JN, dJhyper, PSbej, PSstoch}.

\begin{definition}\label{def:hoeff}
		Let $(X_1, \ldots, X_n)$ be a family of independent random variables,  and $(\mathcal{F}_J)_{J\subseteq [n]}$ be a family of  $\sigma$-fields defined by
		$$\mathcal{F}_J\defeq \sigma(X_j : j\in J), ~~~~J\subseteq [n].$$
		A centered $\mathcal{F}_{[n]}$-measurable  random variable admits the Hoeffding decomposition whenever it might be expressed as
		$$W_n = \sum_{J\subseteq [n]}W_J,$$ 
		where every $W_J$, $J\subseteq [n]$ is a  $\mathcal{F}_J$-measurable random variable and  
		$$\mathbb{E}[W_J \mid \mathcal{F}_K] = 0, ~~~~\text{if}~~~~J\nsubseteq K\subseteq [n].$$
		In particular,  if  for some fixed $d\in[n]$ it holds that $W_J=0$ a.s. for $|J|>d$ and there exists $J\subseteq[n]$ such that $|J|=d$ and $\Var[W_J]>0$,  we call $W_n$ a generalized $U$-statistic of order $d$.
	\end{definition}
	Let us denote by  $e_1, e_2, \ldots, e_{2^n-1}$  all the edges of the complete hypergraph $K^*_n$. Then, let $X_1, X_2, \ldots, X_{2^n-1}$ be independent random variables representing the existence of these edges, such that for any $i\in[2^n-1]$   $X_i$ is a random variable with the Bernoulli distribution $\mathcal{B}(1, p_{|e_i|})$.
	Notice that  hypergraphs without isolated vertices are uniquely determined by their edges. Thus, for such a hypergraph $H$ we have (cf. Remark \ref{rem:isolated})
	\begin{equation}\label{form:copy2}
		Z_H^n=\sum_{A\subseteq [2^n-1]} \I _{A \sim H}  \prod_{a\in A}X_a,
	\end{equation}
	where ${A \sim H}$ means that  the edges $\{e_a:a\in A\}$ constitute a hypergraph which is isomorphic to $H$. Clearly, only the terms corresponding to the sets $A$ satisfying  $|A|=e_H$ do not vanish.
	\begin{theorem}\label{theo:hoeff}
		For a hypergraph $H$ without isolated vertices  the random variable $\widetilde Z_H^n$ admits the Hoeffding decomposition
	\begin{align*}
\widetilde Z_H^n=\sum_{\substack{B\subseteq [2^n-1]\\1\leq |B|\leq e_H}} W_B,
	\end{align*}
		where 
		\begin{align}\label{eq:W_B}
		W_B=\frac{P_H}{\sqrt{\mathrm{Var} [Z_H^n]}} \[ \prod_{b\in B}\(\widetilde X_b\sqrt{\frac{1-p_{|e_b|}}{p_{|e_b|}}}\)\] \sum_{\substack{A\subseteq[2^n-1]\\A\cap B = \emptyset}}\I _{A\cup B \sim H}.
		\end{align}
		
	\end{theorem}
	\begin{proof} The uncorrelatedness  of $W_B$ and $W_{B'}$ for $B\neq B'$ is clear, since every $\widetilde X_b$ is centered.  
		Rewriting  \eqref{form:copy2} we get 
		\begin{align*}
		Z_H^n&=\sum_{A\subseteq [2^n-1]} \I _{A \sim H}  \prod_{a\in A}X_a=\sum_{A\subseteq [2^n-1]} \I _{A \sim H}  \prod_{a\in A}\big[(X_a-p_{|e_a|})+p_{|e_a|}\big]\\
		&=\sum_{A\subseteq [2^n-1]}\I _{A \sim H}  \sum_{B\subseteq A}\( \prod_{b\in B}(X_b-p_{|e_b|})\)\( \prod_{a\in A\backslash B}p_{|e_a|}\)\\
		&=\sum_{A\subseteq [2^n-1]}\I _{A \sim H}  \sum_{B\subseteq A}\( \prod_{b\in B}\frac{X_b-p_{|e_b|}}{p_{|e_b|}}\)\( \prod_{a\in A}p_{|e_a|}\)\\
			&=P_H\sum_{A\subseteq [2^n-1]}\I _{A \sim H}  \sum_{B\subseteq A} \prod_{b\in B}\(\widetilde X_b\sqrt{\frac{1-p_{|e_b|}}{p_{|e_b|}}}\)\\
			&=P_H\sum_{B\subseteq [2^n-1]}  \prod_{b\in B}\(\widetilde X_b\sqrt{\frac{1-p_{|e_b|}}{p_{|e_b|}}}\) \sum_{\substack{A\subseteq[2^n-1]\\A\cap B = \emptyset}}\I _{A\cup B \sim H}.
		\end{align*}
		Observe now that the term of the sum corresponding to $B=\emptyset$ is
		$P_H  \sum_{A}\I _{A \sim H}=\E\[Z_H^n\]$. Furthermore, due to the last indicator $\I _{A\cup B \sim H}$, we can restrict the size of $B$ to be less or equal $e_H$.
	Thus,  we arrive at
	$$\widetilde Z_H^n=\frac{Z_H^n-\E\[Z_H^n\]}{\sqrt{\mathrm{Var} [Z_H^n]}}=\frac{P_H}{\sqrt{\mathrm{Var} [Z_H^n]}}\sum_{\substack{B\subseteq [2^n-1]\\1\leq |B|\leq e_H}}  \prod_{b\in B}\(\widetilde X_b\sqrt{\frac{1-p_{|e_b|}}{p_{|e_b|}}}\) \sum_{\substack{A\subseteq[2^n-1]\\A\cap B = \emptyset}}\I _{A\cup B \sim H}$$
	as required.
	\end{proof}
	Using the notation from Theorem \ref{theo:hoeff}, we define for $H$ and $F$ without isolated vertices
		\begin{align*}
		I_F&:=\sum_{B\sim F}W_B,\ \ \ \ F\subseteq H,\\
		I_m&:=\sum_{|B|=m}W_B=\sum_{\substack{F\subseteq H\\e_F=m}}I_F,\ \ \ \ \ 1\leq m \leq e_H.
		\end{align*}
This lets us write
		\begin{align} \label{eq:Z=II}
		\widetilde{Z}_H^n=\sum_{F\subseteq H}I_F=\sum_{m=1}^{e_H}I_m.
		\end{align}
		\begin{proposition}\label{theo:VarI_m}
		Let $H$ be a hypergraph without isolated vertices. Then, for any $F\subseteq H$ with at least one edge  and without isolated vertices we have
		\begin{align*}
				\mathrm{Var}[I_F] \approx \frac{P_H^2}{\mathrm{Var} [Z_H^n]}
			n^{2v_H-v_F}\frac{Q_F}{P_F}.
		\end{align*}
		Furthermore, for any hypergraph $H$ it holds that  
		\begin{align}\label{eq:VarZ}
				\mathrm{Var}[Z_H^n] &\approx
				P_H^2n^{2v_H}\max_{F\subseteq H, e_F \geq1} n^{-v_F} \frac{Q_F }{P_F}={P_H^2n^{2v_H}}\({\min_{F\subseteq H, e_F \geq1} n^{v_F} P_F/{Q_F }}\)^{-1}.
		\end{align}
	\end{proposition}
	\begin{proof} 
	From \eqref{eq:W_B} we have
	\begin{align*}
	\mathrm{Var}[I_F]=\sum_{B\sim F}\frac{P^2_H}{{\mathrm{Var} [Z_H^n]}}  \prod_{b\in B}\({\frac{1-p_{|e_b|}}{p_{|e_b|}}}\) \bigg(\sum_{\substack{A\subseteq[2^n-1]\\A\cap B = \emptyset}}\I _{A\cup B \sim H}\bigg)^2.
	\end{align*}
	Since $\sum_A\I _{A\cup B \sim H}\approx n^{v_H-v_F}$ when $B\sim F$,  and $\sum_{B\sim F}1\approx n^{v_F}$, we estimate
	\begin{align*}
	\mathrm{Var}[I_F]&\approx n^{v_F}\frac{P^2_H}{{\mathrm{Var} [Z_H^n]}} \frac{Q_F}{P_F} \(n^{v_H-v_F}\)^2=\frac{P_H^2}{\mathrm{Var} [Z_H^n]}
			n^{2v_H-v_F}\frac{Q_F}{P_F},
	\end{align*}
		as required. 
		
		Let us turn our attention to the latter part of the assertion. For hypergraphs without isolated vertices, it  follows from orthogonality of the variables $I_F$ and the equality
		$$\sum_{\substack{F\subseteq H\\e_F\geq1}}\mathrm{Var}[I_F]=\mathrm{Var}[\widetilde Z_H^n]=1.$$ 
		It remains to prove the assertion for $H$ with isolated vertices. Creating such a hypergraph $H$ by adding $k$ vertices to  $H'$ without isolated vertices, we simply have $Z_{H}^n={n-v_{H'}\choose k}Z_{H'}^n$ and hence
		\begin{align*}
		\Var[Z_{H}^n]&={n-v_{H'}\choose k}^2\Var[Z_{H'}^n]\approx n^{2(v_{H}-v_{H'})}P_{H'}^2\max_{F\subseteq H': e_F \geq1} n^{2v_{H'}-v_F} \frac{Q_F }{P_F}\\
		&\approx P_{H}^2\max_{F\subseteq H, e_F \geq1} n^{2v_{H}-v_F} \frac{Q_F }{P_F},
		\end{align*}
		where we used the fact that  the last maximum is realised by a subhypergraph $F$ without isolated vertices.
	\end{proof}
	
	\begin{remark}\label{rem:var}Due to inhomogeneity of the model $\rh(n, \p)$, the maximum in \eqref{eq:VarZ} is rather involved. One may understand it better by observing that if we add to a subhypergraph $F$ an edge whose probability is greater than $1/2$, then the expression $n^{-v_F} \frac{Q_F }{P_F}$ decreases. Thus, it is sufficient to consider subhypergraphs with minimal number of such edges, which leads to 
	\begin{align}
	\mathrm{Var}[Z_H^n]& \approx
				P_H^2n^{2v_H}\(\max\limits_{\substack{F\subseteq H,\, e_F \geq1\\p_{|e|}\leq1/2\text{ for }e\in E(F)}}  \frac{1}{n^{v_F}P_F}
				+\max_{e\in E(H):\, p_{|e|}>1/2} \frac{1-p_{|e|}}{n^{|e|}}\)\\& \approx
				P_H^2n^{2v_H}\(\max\limits_{\substack{F\subseteq H,\, e_F \geq1\\p_{|e|}\leq1/2\text{ for }e\in E(F)}}  \frac{1}{n^{v_F}P_F}
				+\max_{e\in E(H)} \frac{1-p_{|e|}}{n^{|e|}}\),
				\end{align}
				where in the last maximum one can include terms corresponding to $p_{|e|}\leq1/2$, since they are dominated by the terms from the previous maximum. 
	\end{remark}
	
Let us note that in \cite[proof of Theorem 3.2]{Dewar} one can find an  estimate of the variance of a somewhat different  form
$$
			\mathrm{Var}[Z_H^n] \approx P_H^2\max_{F\subseteq H, e_F \geq 1} n^{2v_H-v_F} \frac{1 - P_F }{P_F}.
$$
	Nevertheless, 	both of them are equivalent, even though $1-P_F\approx \max_{e\in E(F)}\{1-p_{|e|}\}$ is generally not comparable with $Q_F$. 
	
	We end this subsection with an observation that if the first sum in \eqref{eq:Z=II} is restricted to subhypergraphs of some $H'\subseteq H$,  then we obtain  the  variable $\widetilde Z^n_{H'}$ multiplied by some constant.
		\begin{proposition}\label{prop:subHoeff}
		For the decomposition \eqref{eq:Z=II} and for any $H'\subseteq H$ without isolated vertices we have $$
		\sum_{F\subseteq H'} I_F=\alpha_n(H', H) \widetilde Z_{H'}^n
		$$
		with 
		\begin{align*}
		\alpha_n(H',H)= C_{H'}{n- v_{H'}\choose v_H-v_{H'}}\frac{P_{H}\sqrt{\mathrm{Var} [Z_{H'}^n]}}{P_{H'}\sqrt{\mathrm{Var} [Z_{H}^n]}},
		\end{align*}
		where $C_{H'}$ denotes the number of hypergraphs isomorphic to $H$ that might be built on $H'$ and some given $v_H-v_{H'}$ {other} vertices. Furthermore, the following estimate holds 
		\begin{align}\label{eq:estalpha}
		\alpha_n(H',H)\approx \sqrt{\frac{\max\limits_{F\subseteq H', e_F \geq1} n^{-v_F} {Q_F }/{P_F}}{\max\limits_{F\subseteq H, e_F \geq1} n^{-v_F} {Q_F }/{P_F}}}.
		\end{align}
		\end{proposition}
		\begin{proof}

		Let $F\subseteq H'$ be a  subhypergraph of $H'$ without isolated vertices. Then, for any $[2^n-1]\supseteq B\sim F$  we have
		\begin{align*}
		\sum_{\substack{A\subseteq[2^n-1]\\A\cap B = \emptyset}}\I _{A\cup B \sim H}&=\sum_{\substack{A_1\subseteq[2^n-1]\\A_1\cap B = \emptyset}}\I_{A_1\cup B\sim H'}\sum_{\substack{A_2: A_1\cap A_2=\emptyset,\\(A_1\cup A_2)\cap B =\emptyset}}\I _{\(A_2\cup A_1\cup B\) \sim H}\\
		&=C_{H'}{n- v_{H'}\choose v_H-v_{H'}}\sum_{\substack{A_1\subseteq[2^n-1]\\A_1\cap B = \emptyset}}\I_{A_1\cup B\sim H'}.
		\end{align*}
		Thus, by \eqref{eq:W_B} we get
		\begin{align*}
		\sum_{F\subseteq H'}I_F
		&=\sum_{F\subseteq H'}\sum_{B\sim F}\frac{P_H}{\sqrt{\mathrm{Var} [Z_H^n]}}  \prod_{b\in B}\(\widetilde X_b\sqrt{\frac{1-p_{|e_b|}}{p_{|e_b|}}}\) \sum_{\substack{A\subseteq[2^n-1]\\A\cap B = \emptyset}}\I _{A\cup B \sim H}\\
		&=\frac{P_{H}\sqrt{\mathrm{Var} [Z_{H'}^n]}}{P_{H'}\sqrt{\mathrm{Var} [Z_{H}^n]}}\sum_{F\subseteq H'}\sum_{B\sim F}\frac{P_{H'}}{\sqrt{\mathrm{Var} [Z_{H'}^n]}}  \prod_{b\in B}\(\widetilde X_b\sqrt{\frac{1-p_{|e_b|}}{p_{|e_b|}}}\) C_{H'}{n- v_{H'}\choose v_H-v_{H'}}\sum_{\substack{A\subseteq[2^n-1]\\A\cap B = \emptyset}}\I _{A\cup B \sim H'}\\
		&=C_{H'}{n- v_{H'}\choose v_H-v_{H'}}\frac{P_{H}\sqrt{\mathrm{Var} [Z_{H'}^n]}}{P_{H'}\sqrt{\mathrm{Var} [Z_{H}^n]}}\widetilde Z_{H'}^n,
		\end{align*}
		as required. Next, by \eqref{eq:VarZ}, we estimate
		\begin{align*}
		\alpha_n(H,H')&\approx n^{v_H-v_{H'}}\frac{P_H}{P_{H'}}\sqrt{\frac{P_{H'}^2\max\limits_{F\subseteq H', e_F \geq1} n^{2v_{H'}-v_F} {Q_F }/{P_F}}{P_H^2\max\limits_{F\subseteq H, e_F \geq1} n^{2v_H-v_F} {Q_F }/{P_F}}}
		=\sqrt{\frac{\max\limits_{F\subseteq H', e_F \geq1} n^{-v_F} {Q_F }/{P_F}}{\max\limits_{F\subseteq H, e_F \geq1} n^{-v_F} {Q_F }/{P_F}}},
		\end{align*}
which ends the proof.
		\end{proof}

\subsection{Notation }
Below,  we gather the most common symbols and notations appearing in the article. The other ones  occur rather locally and are accompanied with their definitions.   
\begin{itemize}
\item$[n]=\{1, 2, \ldots, n\}$.
\item$K^*_n=([n], 2^{[n]}\backslash \{\emptyset\})$ - a complete hypergraph on  $[n]$.
\item $V(H)$ - set of vertices of $H$.
\item $E(H)$ - set of edges of $H$.
\item $v_H$ - number of vertices of $H$.
\item $e_H$ - number of edges of $H$.
\item $|e|$ - size of the edge $e$.
\item$N_H^n=|\{H'\subseteq K^*_n:H' \simeq H\}|$  - number of isomorphic copies of $H$ in $K_n^*$.
\item $\I_{H'}=\I_{\{H'\subseteq \rh(n,  \p )\}}$ for some $H'\subseteq K^*_n$.
\item $Y_{H'}=\I_{H'}-\E[\I_{H'}]$ for some $H'\subseteq K^*_n$.
\item $Z_H^n=\sum_{\substack{H'\subseteq K^*_n \\ H' \simeq H}} \I_{H'}$ - number of isomorphic copies of $H$ in $\rh(n,   \p  )$.
\item $\widetilde X=(X-\E[X])/\sqrt{\Var[X]}$ - standardisation of a random variable $X$.
%\item $V_H^n=N_H^n-Z_H^n$ -  number of isomorphic copies of $H$ that do not exist  in  $\rh(n, \mathbf  p)$
\item $P_H= \prod_{e\in E(H)}p_{|e|}$.
\item $Q_H= \prod_{e\in E(H)}(1-p_{|e|})$.
\item $f_n\lesssim g_n$ - for two functional  sequences: $\exists_{C>0, n_0\in\N}\forall _{n\geq n_0}\,0\leq f_n\leq Cg_n$. \\[5pt]
 \phantom{$f_n\lesssim g_n$ -\ }The constants  $C$ and $n_0$ may depend on $H$. If they {also} depend on 
 another   \mbox{\phantom{$f_n\lesssim g_n$ -\ }parameter} $a$, we write $\stackrel{a}{\lesssim}$.

\item$f_n \approx g_n \Leftrightarrow f_n\lesssim g_n\lesssim f_n$. Similarly $f_n \stackrel{a}\approx g_n\Leftrightarrow f_n\stackrel{a}\lesssim g_n\stackrel{a}\lesssim f_n$.

\item $F\subseteq H \Leftrightarrow V(F)\subseteq V(H)$ and $E(F)\subseteq E(H)$.
\item $F\simeq H \Leftrightarrow$ $F$ is isomorphic to $H$.
\item $A\sim H$ - for a set $A\subseteq [2^n-1]$ and a hypergraph $H\subseteq K_n^*$, when edges numbered by \mbox{\phantom{$A\sim H$ -\ }all elements} from $A$ create a hypergraph isomorphic to $H$.
\item $\text{Lip}(h)$  the   Lipschitz constant of a function $h:\R\rightarrow\R$.
\item$\mathcal N$ - standard normal distribution or a random variable with such a distribution.
\item $d_W$ - Wasserstein distance, see \eqref{eq:d_W}.
\item $d_K$ - Kolmogorov distance, see \eqref{eq:d_K}.
\item $d_{W/K}=\max\{d_W,\,d_K\}$.
\item $W_B$ - a term of the Hoeffding decomposition of $\widetilde Z_H^n$, see Theorem \ref{theo:hoeff}.
\item edgewise separable (e.sep.), edgewise inseparable (e.ins.) - see Definition \ref{def:sep}.
\end{itemize}

\section{Main results}

\subsection{General case}
A typical starting point in research  on properties of random graphs is the threshold function. In our setting, the edge existence probability is a vector, and therefore we provide a threshold-like result in the following, a bit implicit,  but very simple manner.
\begin{proposition}\label{prop:threshold} For a fixed hypergraph $H$ we have
\begin{align}
			\mathbb{P}(Z_H^n > 0) \overset{n\rightarrow \infty}{\longrightarrow} \left\{
			\begin{array}{lll}
				1, & \text{when }&  \min\limits_{F\subseteq H:\,e_F\geq1}\{n^{v_F}\hspace{-5pt} \prod\limits _{e\in E(F)}p_{|e|}\}\rightarrow \infty,  \\[10pt]
				0, & \text{when} &\min\limits_{F\subseteq H:\,e_F\geq1}\{n^{v_F}\hspace{-5pt} \prod\limits _{e\in E(F)}p_{|e|}\}\rightarrow 0.
			\end{array}
			\right.
		\end{align}
\end{proposition}
The above proposition is a slightly  improved version of Theorem 3.2 in \cite{Dewar}.  The first condition is unchanged and the latter one is obtained by Markov's  inequality applied as follows
\begin{align}
\P\(Z_H^n>0\)\leq\min\limits_{F\subseteq H:\,e_F\geq1}\P\(X_F>0\){\leq} \min\limits_{F\subseteq H:\,e_F\geq1}\E\[X_F\]\approx \min\limits_{F\subseteq H:\,e_F\geq1}\{n^{v_F}\hspace{-5pt} \prod\limits _{e\in E(F)}p_{|e|}\},
\end{align}
where  minima were simply added in comparison to the original proof in \cite{Dewar}.

Next, we turn our attention to the conditions of the asymptotic normality. Theorem \ref{thm:conditions}  generalizes \eqref{eq:ERns} as well as the   result for uniform hypergraphs from \cite{dJhyper}.
	\begin{theorem}\label{thm:conditions}
		Let $H$  be any hypergraph. 
	Then $\widetilde{Z}_H^n \overset{d}{\rightarrow} \mathcal{N}$
		holds if and only if

$$ \min\limits _{F\subseteq H:\, e_F\geq1}P_F\,n^{v_F}\rightarrow \infty$$
		 and 
		\begin{equation}\label{long_necessary}
		\forall_{e\in E(H)}\,\({n^{|e|}(1-p_{|e|})} + \frac{\max_{F\subseteq H:\,e_F\geq1} n^{-v_F}Q_F/P_F} {n^{-{|e|}} (1-p_{{|e|}})}\){\rightarrow}\, \infty.
		\end{equation}
	\end{theorem}
The last quotient of \eqref{long_necessary} looks a bit mysterious, so let us make some  comments on it. First of all, it matters only for $p_{|e|}$ close to $1$, since else $n^{|e|}(1-p_{|e|})\rightarrow\infty$  and the whole expression in parentheses tends to $\infty$ anyway. Thus, for  $p_{|e|}>1/2$,  we may estimate the said quotient by $\Var[Z_H^n]/\Var[X_{\{e\}}]$  (see Proposition \ref{theo:VarI_m}), where $\{e\}$ stands for the hypergraph consisting of the edge $e$ and its vertices only. Additionally, since $\E[Z^n_F]{\approx}P_Fn^{v_F} $ for $F\subseteq H$, we may rewrite the conditions from Theorem \ref{thm:conditions} probabilistically  as follows
$$ \min\limits _{F\subseteq H:\, e_F\geq1}\E[Z_F^n]\rightarrow \infty
		 \ \ \text{ and }\ \  
		\forall_{e\in E(H)}\[{n\choose |e|}-\E[Z_{\{e\}}^n] \] +\frac{\Var[Z_H^n]}{\Var[Z_{\{e\}}^n]}\ {\longrightarrow} \ \infty.
$$
This gives us some ideas of interpretation.	By Proposition \ref{prop:threshold}, the first condition ensures at least one copy of $H$	with high probability. The limit $\[{n\choose |e|}-\E[Z_{\{e\}}^n] \]\rightarrow \infty$ means that the number of edges of size $|e|$ that do not exist tends to infinity, so that the subhypergraph of $\rh(n, \p)$ consisting of the  edges of that size is not to close to the deterministic complete $|e|$-uniform hypergraph. However, even in the case it is too close, we can hope for randomness coming from other edges, which may happen if $\Var[Z_{\{e\}}^n]$ has asymptotically no impact on $\Var[Z_H^n]$. 

Next, we establish bounds on the Wasserstein distance
\begin{align}\label{eq:d_W}
d_{W}(\widetilde{Z}_H^n, \mathcal{N}):=\sup_{\text{Lip}(h)\leq1}\big|\E[h(\widetilde{Z}_H^n)-h(\mathcal N)]\big|,
\end{align}
where $\text{Lip}(h)$ denotes the   Lipschitz constant of a function $h:\R\rightarrow\R$, and the Kolmogorov distance
\begin{align}\label{eq:d_K}
d_{K}(\widetilde{Z}_H^n, \mathcal{N}):=\sup_{t\in \R}\big|\P\(\widetilde{Z}_H^n\leq t\)-\P(\mathcal N\leq t)]\big|
\end{align}
between $\widetilde{Z}_H^n$ and a random variable  $\mathcal N$ with the standard normal distribution. In the sequel, we additionally use the notation $d_{W/K}:=\max\{d_W, d_K\}$. Note also that the distance $d_1$ defined in \eqref{eq:d_1} is smaller than the Wasserstein distance, so there is no need to consider it separately.
\begin{theorem}\label{theo:mainbounds}
		For any hypergraph  $H$  we have 
		 $$d_{W}(\widetilde{Z}_H^n, \mathcal{N})\lesssim\frac1{\({{ \min\limits_{\substack{F\subseteq H:\,e_F\geq1}}{P_F}\,n^{v_{F}}}}\)^{1/4}}+\sum_{e\in E(H)}\(\frac{n^{-|e|} (1-p_{|e|})}{\max\limits_{F\subseteq H:\,e_F\geq1} n^{-v_F}Q_F/P_F}\wedge \frac1{(1-p_{|e|})n^{|e|}}\)^{1/2}$$
		and
		 	 $$d_{K}(\widetilde{Z}_H^n, \mathcal{N})\lesssim\frac1{\({{ \min\limits_{\substack{F\subseteq H:\,e_F\geq1}}{P_F}\,n^{v_{F}}}}\)^{1/5}}+\sum_{e\in E(H)}\(\frac{n^{-|e|} (1-p_{|e|})}{\max\limits_{F\subseteq H:\,e_F\geq1} n^{-v_F}Q_F/P_F}\)^{1/3}\hspace{-8pt}\wedge \(\frac1{(1-p_{|e|})n^{|e|}}\)^{1/2}\hspace{-2pt}.$$
		 	 Additionally, if $p_{|e|}<c$ for some $c\in (0,1)$ and all $e\in E(H)$, we have
		 	 \begin{align}\label{aux8}
		 	 d_{W/K}(\widetilde{Z}_H^n, \mathcal{N})\lesssim\frac1{\({{ \min\limits_{\substack{F\subseteq H:\,e_F\geq1}}{P_F}\,n^{v_{F}}}}\)^{1/2}}.
\end{align}
	\end{theorem}
	\begin{remark} The first two bounds correspond to convergence conditions and have no limitations on probabilities  $p_{|e|}$. On the other hand, we pay a price  for this universality and for some ranges of parameters the rates are not optimal, which might be verified for instance in the case of $2$-edge count with $p_2=1/n$ (cf. Theorem \ref{thm:4mp}), or simply by observing that  the power of the first term  shifts from $1/2$ in \eqref{eq:d_1} and \eqref{aux8}  to $1/4$ for Wasserstein distance and to $1/5$ for the Kolmogorov one in the general case.  Nevertheless, these bounds are still significantly better than the ones obtained directly by the methods applied earlier to the $\mathbb G(n,p)$ model. The latter one may  diverge to infinity even when $\widetilde{Z}_H^n$ is asymptotically normal.   See Example \ref{example1} for more details.
	
	\end{remark} 
%	\begin{remark} 
%	The proof of Theorem \ref{theo:mainbounds} relies on the modified  methods developed  in \cite{PSbej,PSejp2}. In Example \ref{example1} we show that direct application of general bounds from \cite{PSejp2}, as well as the ones from \cite{JLR, ER}, leads to significantly worse results.
%	\end{remark}
	
	\subsection{Homogeneous case}
	\label{sec:homo}
	In the case when   $  \p =(p, p, \ldots , p)$, we not only simplify the results, but improve them and derive some new ones as well. In particular, we are in position to incorporate the {\it fourth moment phenomenon}. First, to identify properly the  quantities appearing below,  let us observe that estimates \eqref{eq:Eestgen} and \eqref{eq:Vestgen} reduce to the form 
\begin{align}\label{eq:homoE}
\E[{Z}_H^n]&\approx p^{e_H}n^{v_H},\\\label{eq:homoVar}
{\Var}[{Z}_H^n]&\approx{(1-p)n^{2v_H}p^{2e_H}}\({\min\limits_{\substack{F\subseteq H:\,e_F\geq1}}p^{e_F}n^{v_F}}\)^{-1}.
\end{align}
The conditions for asymptotic normality of $\widetilde{Z}_H^n$ are as follows.
		\begin{theorem}\label{thm:homoequiv}
	 In the homogeneous model $\rh(n, p)$, the following are equivalent
	\begin{enumerate}
	\item[1)] $\widetilde{Z}_H^n \overset{d}{\rightarrow} \mathcal{N}$,
	\item[2)]\label{item2} $\min\limits_{\substack{F\subseteq H:\,e_F\geq1}}p^{e_F}n^{v_F}\rightarrow\infty$ and $(1-p)n^{\min\{|e|:e\in E(H)\}}\rightarrow\infty$,
	\item[3)] $(1-p)\min\limits_{\substack{F\subseteq H:\,e_F\geq1}}p^{e_F}n^{v_F}\rightarrow\infty$,
	\item[4)] $\E\[\(\widetilde{Z}_H^n\)^4 \]\rightarrow 3$.
	\end{enumerate}
	\end{theorem}
	Although the proofs are moved to subsequent sections, we would like to present here the proof of the implication {\it 1) $\Rightarrow$ 2)}. The reason is that it is very short and significantly simpler than the known  proofs  in the case of  the binomial model $\mathbb G(n, p)$ \cite {Ruc88, JLR}.  
		\begin{proof}[Proof of 1) $\Rightarrow$ 2) in Theorem \ref{thm:homoequiv}] The first condition in  {\it 2)} is obvious  for $p>1/2$, we therefore consider $p\leq 1/2$.  Due to ${Z}_H^n\geq0$ we have
	\begin{align}\label{aux11}
	\widetilde{Z}_H^n =\frac{{Z}_H^n-\E[{Z}_H^n]}{\sqrt{\mathrm{Var}[{Z}_H^n]}}\geq- \frac{\E[{Z}_H^n]}{\sqrt{\mathrm{Var}[{Z}_H^n]}}.
	\end{align}
	Since $\widetilde{Z}_H^n \overset{d}{\rightarrow} \mathcal{N}$ and the distribution  $\mathcal{N}(0,1)$ is supported on the whole real line, it follows that the right-hand side of \eqref{aux11} tends to $-\infty$, and consequently $\frac{\E[{Z}_H^n]}{\sqrt{\mathrm{Var}[{Z}_H^n]}}\rightarrow\infty$. Then, the estimates \eqref{eq:homoE} and \eqref{eq:homoVar}  give us ${(\E[{Z}_H^n])^2}/{\mathrm{Var}[{Z}_H^n]}\approx \min_{\substack{F\subseteq H:\,e_F\geq1}}p^{e_F}n^{v_F}$, which implies the first condition in {\it 2)}. 
	
	Similarly, in the case of the latter condition in {\it 2)} it suffices to focus on $p>1/2$. Let us denote by $V^n_H$ the number of all hypergraphs isomorphic to $H$ that do not occur in $\mathbb H(n,p)$, i.e.,
	$$V_H^n:=N_H^n-{Z}_H^n,$$
	where $N_H^n=|\{H'\in K^*_n:H'\simeq H\}|$ is the number of all subhypergraphs of the complete hypergraph $K^*_n$ that are isomorphic to $H$. Since $V_H^n\geq0$ and $\widetilde V_H^n=-\widetilde{Z}_H^n  \overset{d}{\rightarrow} \mathcal{N}$, repeating the argument from the first part of the proof, we get ${\E[{V}_H^n]}/{\sqrt{\mathrm{Var}[{V}_H^n]}}\rightarrow\infty$. Finally, using \eqref{eq:homoVar} and the assumption $p>1/2$ we estimate
	\begin{align*}
	\frac{\E[{V}_H^n]}{\sqrt{\mathrm{Var}[{V}_H^n]}}&=\frac{\E\[\sum_{\substack{H'\subseteq K^*_n\\H'\simeq H}}\(1-\I_{\{H'\subseteq\mathbb H(n,p)\}}\)\]}{\sqrt{\mathrm{Var}[{Z}_H^n]}}=\frac{N_H^n(1-p^{e_H})}{\sqrt{\mathrm{Var}[{Z}_H^n]}}\\[6pt]
%	&\approx n^{v_H}(1-p)\sqrt{\frac{\min\limits_{\substack{F\subseteq H:\,e_F\geq1}}p^{e_F}n^{v_F}}{{{(1-p)n^{2v_H}p^{2e_H}}}}}
	&\approx n^{v_H}(1-p)\sqrt{\frac{\min\limits_{\substack{F\subseteq H:\,e_F\geq1}}n^{v_F}}{{{(1-p)n^{2v_H}}}}}
	=\sqrt{(1-p)n^{\min\{|e|:e\in E(H)\}}},
	\end{align*}
	which ends the proof.
	\end{proof}
	Next, we present convergence rates corresponding to the asymptotic normality conditions.	
		\begin{theorem}\label{thm:4mp}
	For any hypergraph $H$ we have
	\begin{align}
	d_{W/K}(\widetilde{Z}_H^n, \mathcal{N})&\lesssim \((1-p){{ \min\limits_{\substack{F\subseteq H:\,e_F\geq1}}p^{e_F}\,n^{v_{F}}}}\)^{-1/2}\\\label{eq:4mp1}
	&\approx {\left|\E\[\(\widetilde{Z}_H^n\)^4 \]-3\right|}^{1/2}+{n^{-\min\{|e|:e\in E(H)\}/2}}.
	\end{align}
	 Additionally, there exists $\delta\in (0,1/2)$ such that for $p\in(0,\delta)\cup(1-\delta,1)$ the last term as well as the absolute value might be omitted, i.e., it holds that 
	\begin{align}\label{eq:4mp2}
	\({{ (1-p)\min\limits_{\substack{F\subseteq H:\,e_F\geq1}}p^{e_F}\,n^{v_{F}}}}\)^{-1}\approx {\E\[\(\widetilde{Z}_H^n\)^4 \]-3}.
	\end{align}
	\end{theorem}
	\begin{remark} In view of \eqref{eq:4mp2}, a  question that arises is  whether the term ${n^{-\min\{|e|:e\in E(H)\}/2}}$ in \eqref{eq:4mp1} is necessary. In general it indeed is necessary, which might be seen in the case of $H$ being a single edge of size $2$. $Z_H^n$ follows then the binomial distribution $B\({n\choose 2}, p\)$ and it is well known that for ${ p=1/2 \pm {\sqrt {1/12}}}$ the binomial distribution is mesokurtic, which exactly means that 
	$$\E\[\(\widetilde{Z}_H^n\)^4 \]-3=0,$$
	while for such $p$'s it holds that  
	$$\({{(1-p) \min\limits_{\substack{F\subseteq H:\,e_F\geq1}}{P_F}\,n^{v_{F}}}}\)^{-1/2}\approx \frac1n=n^{-\min\{|e|:e\in E(H)\}/2}.$$
	\end{remark}

		\begin{remark}\label{rem:isolated} 
		Let $H$ be a fixed hypergraph with $k$ isolated vertices, and denote by $H'$ the graph created by removing those vertices from $H$. Then we have $Z_H^n={n-v_{H'}\choose k}Z_{H'}$, and consequently
		$\widetilde Z_H^n=\widetilde Z_{H'}$.
		One can also verify that all minima and maxima appearing in the results of this section stay unchanged, regardless of whether we take them with respect to $H$ or $H'$. This suffices to prove all the results for $H$ without isolated vertices, which we do, as it allows us to identify  a hypergraph by its edges only.  Note that this observation has been already made  in \cite{ER}, see Remark 4.6 therein.
	\end{remark}

		\section{Bounds for distances}
		
		Let us recall the following bound on both: the Wasserstein and the Kolmogorov distance between a normally distributed random variable $\mathcal N\sim \mathcal N(0,1)$ and a {random} variable given by its Hoeffding decomposition.

	\begin{theorem}[Theorem 4.1 in \cite{PSejp2}]\label{theo:priv2022}
	Let  $1\leq d\leq n$. For any generalized $U$-statistic  $W\in L^4(\Omega)$ of order $d$ admitting the Hoeffding decomposition
from Definition \ref{def:hoeff}, and such that $\E[W^2]=1$, we have
\begin{align}\label{eq:priv2022}
d_{W/K}(W, \mathcal{N}(0, 1))
				\leq C_d\sqrt{S_1 + S_2 + S_3},
\end{align}
				where  $C_d$ depends only on  $d$ and 
		\begin{equation}
			\begin{aligned}
			S_1&=\sum_{0\leq l< i\leq d}\sum_{|A|=i-l}\mathbb{E}\Big[\Big(\sum_{|B|=l, B\cap A=\emptyset}\mathbb{E}[W_{A\cup B}^2|\mathcal{F}_A]\Big)^2\Big],\\
				S_2&=\sum_{1\leq l< i\leq d}\sum_{\substack{|A_1|=|A_2|=i-l\\A_1\cap A_2=\emptyset}}\mathbb{E}\Big[\Big(\sum_{\substack{|B|=l, B\cap(A_1\cup A_2)=\emptyset}}\mathbb{E}[W_{A_1\cup B}W_{A_2\cup B}|\mathcal{F}_{A_1\cup A_2}]\Big)^2\Big],\\
				S_3&=\sum_{1\leq l< i\leq d}\sum_{|A|=i-l}\mathbb{E}\Big[\Big(\sum_{|B|=l, B\cap A =\emptyset}\mathbb{E}[W_B W_{A\cup B}|\mathcal{F}_A]\Big)^2\Big].
			\end{aligned}
		\end{equation}
		
	\end{theorem}
This is the main tool used  in the proof of the next theorem.
	
	\begin{theorem}\label{theo:sufficient}
		For a hypergraph  $H$ with no isolated vertices we have 
		 \begin{align}\label{eq:sufficient1}
		 d_{W/K}(\widetilde{Z}_H^n, \mathcal{N})&\lesssim\bigg({ \min\limits_{\substack{F\subseteq H\\e_F\geq1}}\frac{P_F}{Q_F}\,n^{v_{F}}}\bigg)^{-1/2}+{ \min\limits_{\substack{F\subseteq H\\,e_F\geq1}}\frac{P_F}{Q_F}\,n^{v_{F}}}{\bigg( \max\limits_{\substack{e\in E(H)\\ p_{|e|}>1/2}}\frac{1-p_{|e|}}{n^{3|e|}}\bigg)^{1/2}}\\[6pt] \label{eq:sufficient2}
&\leq\bigg({ \min\limits_{\substack{F\subseteq H\\e_F\geq1}}{P_F}\,n^{v_{F}}}\bigg)^{-1/2}+\bigg( \max\limits_{\substack{e\in E(H)\\ p_{|e|}>1/2}}\frac1{(1-p_{|e|}){n^{|e|}}}\bigg)^{1/2}.
		 \end{align}
	\end{theorem}
	\begin{proof}
		
		We will apply  Theorem \ref{theo:priv2022} to Theorem \ref{theo:hoeff}  and bound the sums $S_1$, $S_2$ and $S_3$. Recall that 
		
		$$
		W_B=\frac{P_H}{\sqrt{\mathrm{Var} [Z_H^n]}}  \prod_{b\in B}\(\widetilde X_b\sqrt{\frac{1-p_{|e_b|}}{p_{|e_b|}}}\) \sum_{\substack{A\subseteq[2^n-1]\\A\cap B = \emptyset}}\I _{A\cup B \sim H}.
		$$
		Consequently, we have 
		\begin{align*}
		\mathbb{E}[W_{A_1\cup B}W_{A_2\cup B}|\mathcal{F}_{A_1\cup A_2}]&=\frac{P^2_H}{{\mathrm{Var} [Z_H^n]}}  \prod_{j\in A_1\cup B}\sqrt{\frac{1-p_{|e_j|}}{p_{|e_j|}}} \prod_{j\in A_2\cup B}\sqrt{\frac{1-p_{|e_j|}}{p_{|e_j|}}} \prod_{j\in A_1\cup A_2}\widetilde X_{j}\\
		&\ \ \ \ \times\(\sum_{\substack{C:C\cap( A_1\cup B) =\emptyset}}\I _{A_1\cup B\cup C \sim H}\)\(\sum_{\substack{C':C'\cap(A_2\cup B) =\emptyset}}\I _{A_2\cup B\cup C' \sim H}\),
		\end{align*}
which gives us 
\begin{align*}
S_2&=\frac{P^4_H}{(\mathrm{Var} [Z_H^n])^2}\sum_{1\leq l< i\leq e_H}\sum_{\substack{|A_1|=|A_2|=i-l\\A_1\cap A_2=\emptyset}} \prod_{j\in A_1}{\frac{1-p_{|e_j|}}{p_{|e_j|}}} \prod_{j\in A_2}{\frac{1-p_{|e_j|}}{p_{|e_j|}}}\\
&\hspace{15pt}\times\Bigg[\sum_{\substack{\substack{|B|=l\\ K\cap(A_1\cup A_2)=\emptyset}}} \prod_{k\in  B}{\frac{1-p_{|e_k|}}{p_{|e_k|}}}\bigg(\sum_{\substack{C:C\cap (A_1\cup B) =\emptyset}}\I _{A_1\cup B\cup C \sim H}\bigg)\bigg(\sum_{\substack{C':C'\cap (A_2\cup B) =\emptyset}}\I _{A_2\cup B\cup C' \sim H}\bigg)\Bigg]^2\\
&\leq\frac{P^4_H}{(\mathrm{Var} [Z_H^n])^2}\sum_{\substack{K,K'\subseteq K^*_n\\e_K=e_{K'}\geq1\\e_{K\cap K'=0}}}\frac{Q_{K}Q_{K'}}{P_{K}P_{K'}}\\[5pt]
&\hspace{15pt}\times\Bigg[\sum_{\substack{L\subseteq K^*_n:\,e_L\geq1\\ e_{L\cap(K\cup K')}=0}}\frac{Q_{L}}{P_{L}}\bigg(\sum_{\substack{M\subseteq K^*_n\\ e_{M\cap (K\cup L)}=0}}\I _{K\cup L\cup M\simeq H}\bigg)\bigg(\sum_{\substack{M'\subseteq K^*_n\\ e_{M'\cap (K'\cup L)}=0 }}\I _{K'\cup L\cup M'\simeq H}\bigg)\Bigg]^2.
\end{align*}
Next, we will consider cases depending on the structure created by the involved hypergraphs. To do so, for 
	two families  of hypergraphs $H_1, H_2, \ldots, H_m$ and  $H_1', H_2', \ldots, H_m'$, $m\in\mathbb{N}$, we will write
	$$\mathcal{U}(H_1, H_2, \ldots, H_m)=\mathcal{U}(H_1', H_2', \ldots, H_m'),$$
	 whenever they are isomorphic, by which we mean that for any nonempty   $I\subseteq [m]$ we have $\bigcap_{i\in I}H_i\simeq \bigcap_{i\in I}H_i'$. Furthermore, let us note that the union  $K\cup K'\cup L$ of hypergraphs in the sum above may not be isomorphic to any subhypergraph of $H$, but the number of their vertices is at most $2v_{H}$. We will therefore choose the patterns of their structures from the following set
	 $$\mathcal K_H:=\{K\subseteq K^*_{2v_H}: e_K\geq1\text{ and }K\simeq F\text{ for some }F\subseteq  H\}.$$   All this leads us to 
\begin{align*}
S_2&\lesssim\frac{P^4_H}{(\mathrm{Var} [Z_H^n])^2}\sum_{\substack{K_0,K_0', L_0\in \mathcal K_H\\e_{K_0}=e_{K'_0}\geq1,\,e_{L_0}\geq1\\e_{K_0\cap K'_0}=e_{L_0\cap(K_0\cup K'_0)}=0}}\sum_{\substack{K,K'\subseteq K^*_n\\\mathcal{U}(K,K')=\mathcal{U}(K_0,K_0')}}\frac{Q_{K_0}Q_{K_0'}}{P_{K_0}P_{K_0'}}\\
&\hspace{20pt}\times\Bigg[\sum_{\substack{L\subseteq K^*_n\\\mathcal{U}(K,K',L)=\mathcal{U}(K_0,K_0',L_0)}}\frac{Q_{L_0}}{P_{L_0}}\(\sum_{M\subseteq K^*_n}\I _{K\cup L\cup M\simeq H}\)\(\sum_{M'\subseteq K^*_n}\I _{K'\cup L\cup M'\simeq H}\)\Bigg]^2\\[8pt]
&\lesssim\frac{P^4_H}{(\mathrm{Var} [Z_H^n])^2}\sum_{\substack{K_0,K_0', L_0\in \mathcal K_H\\e_{K_0}=e_{K'_0}\geq1,\,e_{L_0}\geq1\\e_{K_0\cap K'_0}=e_{L_0\cap(K_0\cup K'_0)}=0}}n^{v_{K_0\cup K_0'}}\frac{Q_{K_0}Q_{K_0'}}{P_{K_0}P_{K_0'}}\[n^{v_{L_0\backslash(K_0\cup K'_0)}}\frac{Q_{L_0}}{P_{L_0}}n^{v_H-v_{K_0\cup L_0}}n^{v_H-v_{K'_0\cup L_0}}\]^2\\
&=\frac{P^4_H}{(\mathrm{Var} [Z_H^n])^2}\sum_{\substack{K_0,K_0', L_0\in \mathcal K_H\\e_{K_0}=e_{K'_0}\geq1,\,e_{L_0}\geq1\\e_{K_0\cap K'_0}=e_{L_0\cap(K_0\cup K'_0)}=0}}\frac{Q_{K_0}Q_{K_0'}Q_{L_0}^2}{P_{K_0}P_{K_0'}P_{L_0}^2}\,n^{4v_H+v_{K_0\cup K_0'}+2v_{L_0\backslash(K_0\cup K'_0)}-2v_{K_0\cup L_0}-2v_{K'_0\cup L_0}}.
\end{align*}
It is easy to see from the Venn diagram that
$$v_{K_0\cup K_0'}+2v_{L_0\backslash(K_0\cup K'_0)}-2v_{K_0\cup L_0}-2v_{K'_0\cup L_0}\leq -v_{K_0}-v_{L_0}-v_{K'_0\cup L_0}.$$
Applying  the equalities $P_{K_0'}P_{L_0}=P_{K_0'\cup L_0}$ and $Q_{K_0'}Q_{L_0}=Q_{K_0'\cup L_0}$, which are true due to $E\(K_0'\cap L_0\)=\emptyset$, we get 
\begin{align}\nonumber
S_2&\lesssim\frac{P^4_Hn^{4v_H}}{(\mathrm{Var} [Z_H^n])^2}\sum_{\substack{K_0,K_0', L_0\in \mathcal K_H\\e_{K_0}=e_{K'_0}\geq1, e_{L_0}\geq1\\e_{K_0\cap K'_0}=e_{L_0\cap(K_0\cup K'_0)=0}}}\(\frac{Q_{K_0}}{P_{K_0}}n^{-v_{K_0}}\)\(\frac{Q_{L_0}}{P_{L_0}}n^{-v_{L_0}}\)\(\frac{Q_{K_0'\cup L_0}}{P_{K_0'\cup L_0}}n^{-v_{K_0'\cup L_0}}\)\\\label{eq:S_2}
&\lesssim\frac{P^4_Hn^{4v_H}}{(\mathrm{Var} [Z_H^n])^2}\max_{\substack{F\subseteq H\\e_F\geq1}}
\(\frac{Q_{F}}{P_F}n^{-v_{F}}\)^3.
\end{align}
In the case of the sum $S_3$ we proceed analogously. For $ A\cap B =\emptyset$ we have
\begin{align*}
\mathbb{E}[W_B W_{A\cup B}|\mathcal{F}_A]&=\frac{P^2_H}{{\mathrm{Var} [Z_H^n]}}  \prod_{j\in B}\sqrt{\frac{1-p_{|e_j|}}{p_{|e_j|}}} \prod_{j\in A\cup B}\sqrt{\frac{1-p_{|e_j|}}{p_{|e_j|}}} \prod_{j\in A}\widetilde X_{j}\\
		&\ \ \ \ \times\(\sum_{C:C\cap B =\emptyset}\I _{C\cup B \sim H}\)\(\sum_{C':C'\cap (A\cup B) =\emptyset}\I _{A\cup B\cup C' \sim H}\),
\end{align*}
and consequently 
\begin{align}\nonumber
S_3&=\frac{P^4_H}{(\mathrm{Var} [Z_H^n])^2}\sum_{1\leq l< i\leq e_H}\sum_{\substack{|A|=i-l}} \prod_{j\in A}{\frac{1-p_{|e_j|}}{p_{|e_j|}}}\\\nonumber
&\hspace{20pt}\times\[\sum_{\substack{|B|=l\\ B\cap A=\emptyset}} \prod_{k\in  B}{\frac{1-p_{|e_k|}}{p_{|e_k|}}}\(\sum_{C:C\cap B =\emptyset}\I _{B\cup C \sim H}\)\(\sum_{C':C'\cap (A\cup B) =\emptyset}\I _{ A\cup B\cup C' \sim H}\)\]^2\\\nonumber
&\lesssim\frac{P^4_H}{(\mathrm{Var} [Z_H^n])^2}\sum_{\substack{K_0, L_0\in \mathcal K_H\\e_{K_0}, e_{L_0}\geq1\\e_{L_0\cap K_0=0}}}\sum_{\substack{K\subseteq K^*_n\\K\simeq K_0}}\frac{Q_{K_0}}{P_{K_0}}\[\sum_{\substack{L\subseteq K^*_n\\\mathcal{U}(K,L)=\mathcal{U}(K_0,L_0)}}\frac{Q_{L_0}}{P_{L_0}}\(\sum_{\substack{M\subseteq K^*_n\\e_{M\cap K}=0}}\I _{l\cup  M\simeq H}\)\(\sum_{\substack{M'\subseteq K^*_n\\e_{M\cap (K\cup L)}=0 }}\I _{K\cup L\cup M'\simeq H}\)\]^2\\[8pt]\nonumber
&\lesssim\frac{P^4_H}{(\mathrm{Var} [Z_H^n])^2}\sum_{\substack{K_0, L_0\in \mathcal K_H\\e_{K_0}, e_{L_0}\geq1\\e_{L_0\cap K_0=0}}}n^{v_{K_0}}\frac{Q_{K_0}}{P_{K_0}}\[n^{v_{L_0\backslash K_0}}\frac{Q_{L_0}}{P_{L_0}}n^{v_H-v_{L_0}}n^{v_H-v_{K_0\cup L_0}}\]^2\\\nonumber
&=\frac{P^4_Hn^{4v_H}}{(\mathrm{Var} [Z_H^n])^2}\sum_{\substack{K_0, L_0\in \mathcal K_H\\e_{K_0}, e_{L_0}\geq1\\e_{L_0\cap K_0=0}}}\(\frac{Q_{K_0}}{P_{K_0}}n^{-v_{K_0}}\)\(\frac{Q_{L_0}}{P_{L_0}}n^{-v_{L_0}}\)^2\\\label{eq:S_3}
&\lesssim\frac{P^4_Hn^{4v_H}}{(\mathrm{Var} [Z_H^n])^2}\max_{\substack{F\subseteq H\\e_F\geq1}}
\(\frac{Q_{F}}{P_F}n^{-v_{F}}\)^3.
\end{align}

	Eventually, we turn our attention to the sum $S_1$. For $ A\cap B =\emptyset$ we have
	\begin{align*}
	\mathbb{E}[W_{A\cup B}^2|\mathcal{F}_A]&=\frac{P^2_H}{{\mathrm{Var} [Z_H^n]}}  \prod_{j\in A\cup B}{\frac{1-p_{|e_j|}}{p_{|e_j|}}} \prod_{j\in A}\(\widetilde X_{j}\)^2\(\sum_{C:C\cap (A\cup B) =\emptyset}\I _{A\cup B\cup C \sim H}\)^2.
	\end{align*}
Noting that for a random variable $X$ with Bernoulli distribution $\mathcal B(1, p)$ it holds that  $\E[\widetilde X^4]=[p^3+(1-p)^3]/[p(1-p)]\approx 1/[p(1-p)]$, we get 
	\begin{align*}
	S_1&=
	\frac{P^4_H}{{\(\mathrm{Var} [Z_H^n]\)^2}}\sum_{0\leq l< i\leq e_H}\sum_{|A|=i-l}\( \prod_{j\in A}{\frac{1-p_{|e_j|}}{p_{|e_j|}}}\)^2  \prod_{j\in A}\mathbb{E}[\widetilde X_{j}^4]\\
	&\hspace{140pt}\times\[\sum_{|B|=l, B\cap A=\emptyset} \prod_{j\in B}{\frac{1-p_{|e_j|}}{p_{|e_j|}}} \(\sum_{C:C\cap (A\cup B) =\emptyset}\I _{A\cup B\cup C \sim H}\)^2\]^2\\
	&\lesssim\frac{P^4_H}{(\mathrm{Var} [Z_H^n])^2}\sum_{\substack{K_0, L_0\in \mathcal K_H\\e_{K_0}\geq1\\e_{L_0\cap K_0}=0}}\sum_{\substack{K\subseteq K^*_n\\K\simeq K_0}}\frac{\(\frac{Q_{K_0}}{P_{K_0}}\)^2}{P_{K_0}Q_{K_0}}\[\sum_{\substack{L\subseteq K^*_n\\\mathcal{U}(K,L)=\mathcal{U}(K_0,L_0)}}\frac{Q_{L_0}}{P_{L_0}}\Bigg(\sum_{\substack{M\subseteq K^*_n\\e_{M\cap (K\cup L)}=0}}\hspace{-10pt}\I _{K\cup L\cup M\simeq H}\Bigg)^2\,\]^2\\[8pt]
		&\lesssim\frac{P^4_H}{(\mathrm{Var} [Z_H^n])^2}\sum_{\substack{K_0, L_0\in \mathcal K_H\\e_{K_0}\geq1\\e_{L_0\cap K_0}=0}}n^{v_{K_0}}\frac{Q_{K_0}}{P^3_{K_0}}\[n^{v_{L_0\backslash K_0}}\frac{Q_{L_0}}{P_{L_0}}\(n^{v_H-v_{K_0\cup L_0}}\)^2\]^2\\
		&\approx \frac{P^4_Hn^{4v_H}}{(\mathrm{Var} [Z_H^n])^2}\max_{\substack{K_0, L_0\in \mathcal K_H\\e_{K_0}\geq1\\e_{L_0\cap K_0}=0}}\frac{Q_{K_0}}{P^3_{K_0}}\(\frac{Q_{L_0}}{P_{L_0}}\)^2n^{-v_{K_0}-2v_{K_0\cup L_0}}.
	\end{align*}
	Since $Q_{K_0}$ and $Q_{L_0}$ appear in the numerators, one can see that the maximum does not change if we narrow the range of graphs to the ones with minimal number of edges whose existence probability is greater than $1/2$. This allows us to consider subhypergraphs of $H$ with no such an edge (denoted by $\mathcal K_H^0$) or the ones consisting of exactly one edge of that type ($\mathcal K_H^1$). Such an argument has been already brought up when discussing variance of $Z_H^n$, see Remark \ref{rem:var}. Thus, we get 
	\begin{align}
	S_1&\lesssim  \frac{P^4_Hn^{4v_H}}{(\mathrm{Var} [Z_H^n])^2}\bigg(\max_{\substack{K_0, L_0\subseteq \mathcal K_H^0\\e_{K_0}\geq1\\e_{L_0\cap K_0=0}}}\(\frac{Q_{K_0}}{P_{K_0}}n^{-v_{K_0}}\)\(\frac{Q_{K_0\cup L_0}}{P_{K_0\cup L_0}}n^{-v_{K_0\cup L_0}}\)^2\frac1{Q_{K_0}^2}+\max_{\substack{K_0\subseteq \mathcal K_H^1\\e_{K_0}\geq1}}\frac{Q_{K_0}}{P^3_{K_0}}n^{-3v_{K_0}}\bigg)\\\label{eq:S_1}
	&\lesssim  \frac{P^4_Hn^{4v_H}}{(\mathrm{Var} [Z_H^n])^2}\bigg(\max_{\substack{F\subseteq H\\e_F\geq1}}\(\frac{Q_{F}}{P_F}n^{-v_{F}}\)^3+\max_{\substack{e\in E(H)\\ p_{|e|}>1/2}}{(1-p_{|e|})}n^{-3|e|}\bigg).
	\end{align}
	Applying \eqref{eq:S_2}, \eqref{eq:S_3} and \eqref{eq:S_1} to \eqref{eq:priv2022}, we get 
	\begin{align*}
	d_{W/K}(\widetilde{Z}_H^n, \mathcal{N})&\lesssim \(\max_{\substack{F\subseteq H\\e_F\geq1}}\(\frac{Q_{F}}{P_F}n^{-v_{F}}\)+\frac{\max_{\substack{e\in E(H)\\ p_{|e|}>1/2}}{(1-p_{|e|})}n^{-3|e|}}{\max_{\substack{F\subseteq H\\e_F\geq1}}\(\frac{Q_{F}}{P_F}n^{-v_{F}}\)^2}\)^{1/2}\\[8pt]
	&\leq\bigg({ \min\limits_{\substack{F\subseteq H\\e_F\geq1}}\frac{P_F}{Q_F}\,n^{v_{F}}}\bigg)^{-1/2}+{ \min\limits_{\substack{F\subseteq H\\e_F\geq1}}\frac{P_F}{Q_F}\,n^{v_{F}}}{\bigg( \max\limits_{\substack{e\in E(H)\\p_{|e|}>1/2}}\frac{1-p_{|e|}}{n^{3|e|}}\bigg)^{1/2}},
	\end{align*}
	which  is the bound \eqref{eq:sufficient1}. Let us now focus on the other one. First, since  $Q_F\in(0,1)$, we simply estimate
	$$\({ \min\limits_{\substack{F\subseteq H:\,e_F\geq1}}\frac{P_F}{Q_F}\,n^{v_{F}}}\)^{-1/2}\leq \({ \min\limits_{\substack{F\subseteq H:\,e_F\geq1}}{P_F}\,n^{v_{F}}}\)^{-1/2}.$$
	Next, let $\bar e\in E(H)$ be such that $p_{|\bar e|}>1/2$ and  $ \min\limits_{\substack{e\in E(H):\, p_{|e|}>1/2}}\frac{n^{3|e|}}{(1-p_{|e|})}=\frac{n^{3|\bar e|}}{(1-p_{|\bar e|})}$. Then we get 
	\begin{align*}
	&\frac{ \min\limits_{\substack{F\subseteq H:\,e_F\geq1}}\frac{P_F}{Q_F}\,n^{v_{F}}}{\Big( \min\limits_{\substack{e\in E(H):\, p_{|e|}>1/2}}\frac{n^{3|e|}}{(1-p_{|e|})}\Big)^{1/2}}\\
	&\leq  	\frac{\frac{1}{1-p_{|\bar e|}}\,n^{|\bar e|}}{\(\frac{n^{3|\bar e|}}{(1-p_{|\bar e|})}\)^{1/2}}=\frac{1}{\((1-p_{|\bar e|}){n^{|\bar e|}}\)^{1/2}}\leq \frac1{\( \min\limits_{\substack{e\in E(H):\, p_{|e|}>1/2}}(1-p_{|e|}){n^{|e|}}\)^{1/2}}.
	\end{align*}
	This completes the proof.
	\end{proof}
%	
%	
%		\begin{theorem}\label{theo:sufficient}
%		For a hypergraph  $H$ with no isolated vertices we have 
%		 $$d_{W}(\widetilde{Z}_H^n, \mathcal{N})\lesssim\(\frac1{{ \min\limits_{\substack{F\subseteq H:\,e_F\geq1}}{P_F}\,n^{v_{F}}}}\)^{1/4}+\sum_{e\in E(H)}\(\frac{n^{-|e|} (1-p_{|e|})}{\max\limits_{F\subseteq H:\,e_F\geq1} n^{-v_F}Q_F/P_F}\wedge \frac1{(1-p_{|e|})n^{|e|}}\)^{1/2},$$
%		 and
%		 	 $$d_{K}(\widetilde{Z}_H^n, \mathcal{N})\lesssim\(\frac1{{ \min\limits_{\substack{F\subseteq H:\,e_F\geq1}}{P_F}\,n^{v_{F}}}}\)^{1/5}+\sum_{e\in E(H)}\(\frac{n^{-|e|} (1-p_{|e|})}{\max\limits_{F\subseteq H:\,e_F\geq1} n^{-v_F}Q_F/P_F}\)^{1/3}\wedge \(\frac1{(1-p_{|e|})n^{|e|}}\)^{1/2}.$$
%	\end{theorem}
	\begin{proof}[\textbf{Proof of Theorem \ref{theo:mainbounds}}] We start with the bound \eqref{aux8}, as it directly follows from Theorem \ref{theo:sufficient}. Indeed, the assumption  $p_{|e|}<c<1$ implies $(1-p_{|e|})\gtrsim p_{|e|}$ for all $e\in E(H)$, and therefore  \eqref{eq:sufficient2} gives us 
		 \begin{align*}
		 d_{W/K}(\widetilde{Z}_H^n, \mathcal{N})&\lesssim\({ \min\limits_{\substack{F\subseteq H:\,e_F\geq1}}P_F\,n^{v_{F}}}\)^{-1/2}+\( \min\limits_{\substack{e\in E(H):\, p_{|e|}>1/2}}p_{|e|}{n^{|e|}}\)^{-1/2}\\
		 &\leq2\({ \min\limits_{\substack{F\subseteq H:\,e_F\geq1}}P_F\,n^{v_{F}}}\)^{-1/2}.
		 \end{align*}

We pass to the proof of the other two bounds of the theorem. For any $\varepsilon(0,1/2))$ we can decompose the hypergraph  $H$ as
		\begin{align}\label{eq:H^}
		H=H^{\leq 1-\varepsilon}_n\cup H^{>1-\varepsilon}_n, 
		\end{align}
		where the hypergraphs $H^{\leq 1-\varepsilon}_n$ and $H^{>1-\varepsilon}_n$  share no common edge (i.e. $E({H_n^{\leq1-\varepsilon}\cap H_n^{>1-\varepsilon}})=\emptyset$), contain no isolated vertices and
		
		\begin{itemize}
			\item $p_{|e|} \leq 1-\varepsilon$ whenever $e\in E(H^{\leq 1-\varepsilon}_n)$,
			\item $p_{|e|} > 1-\varepsilon$ whenever $e\in E(H^{>1-\varepsilon}_n)$.
		\end{itemize}
		Thus, using Proposition \ref{prop:subHoeff} we  represent $\widetilde{Z}_H^n$ as 
		\begin{align}\label{eq:decompZa}
		\widetilde{Z}_H^n&=\sum_{F\subseteq H^{\leq 1-\varepsilon}_n}I_F+\sum_{e\in E(H^{>1-\varepsilon}_n)}I_{\{e\}}+\sum_{\substack{F\subseteq H:\,e_F\geq2\\ E({F\cap H^{>1-\varepsilon}_n})\neq\emptyset}}I_F\\
		&=\alpha_n(H^{\leq 1-\varepsilon}_n,H)\widetilde Z_{H^{\leq 1-\varepsilon}_n}^n+\sum_{e\in E(H^{>1-\varepsilon}_n)}\alpha_n(\{e\},H)\widetilde Z_{\{e\}}^n+\sum_{\substack{F\subseteq H:\,e_F\geq2\\ E({F\cap H^{>1-\varepsilon}_n})\neq\emptyset}}I_F.
		\end{align}
For any $F\subseteq H$ such that $e_{F}\geq2$ and $ E({F\cap H^{>1-\varepsilon}_n})\neq\emptyset$ we have 
\begin{align*}
\mathrm{Var}[I_F]&\approx \frac{P_H^2}{\mathrm{Var} [Z_H^n]}
			n^{2v_H-v_F}\frac{Q_F}{P_F}\leq\(\frac{P_H^2}{\mathrm{Var} [Z_H^n]}
			n^{2v_H-v_{F\cap H^{\leq 1-\varepsilon}_n}}\frac{Q_{F\cap H^{\leq 1-\varepsilon}_n}}{P_{F\cap H^{\leq 1-\varepsilon}_n}}\)\frac{Q_{F\cap H^{>1-\varepsilon}_n}}{P_{F\cap H^{>1-\varepsilon}_n}}\\
			&\leq\(\frac{P_H^2}{\mathrm{Var} [Z_H^n]}
			\max_{F'\subseteq H: e_{F'}\geq1}n^{2v_H-v_{F'}}\frac{Q_{F'}}{P_{F'}}\)\frac{\varepsilon}{(1-\varepsilon)^{e_H}}\lesssim \varepsilon, 
\end{align*}
		which implies
		\begin{align*}
		\mathrm{Var}\bigg[\sum_{\substack{F\subseteq H:\,e_F\geq2\\ E({F\cap H^{>1-\varepsilon}_n})\neq\emptyset}}I_F\bigg]\lesssim \varepsilon.
		\end{align*}
		Thus, since the variables $\widetilde Z_{H^{\leq 1-\varepsilon}_n}^n$ and $\widetilde Z_{\{e\}}^n$, $e\in E(H^{>1-\varepsilon}_n)$,  are all independent { for all $n\in\N$},   Proposition \ref{prop:independent} gives us 
		\begin{align}\label{eq:aux6}
		d_{\Theta}(\widetilde{Z}_H^n,\,\mathcal N)&\lesssim \bigg([\alpha_n(H^{\leq 1-\varepsilon}_n,H)]^{2\gamma}\wedge d_{\Theta}(\widetilde Z_{H^{\leq 1-\varepsilon}_n}^n,\,\mathcal N)\bigg)+\varepsilon^{\gamma}\\[2pt]
		&\ \ \ \ +\sum_{e\in E(H^{> 1-\varepsilon}_n)}\bigg([\alpha_n(\{e\},H)]^{2\gamma}\wedge d_{\Theta}(\widetilde Z_{\{e\}}^n,\,\mathcal N)\bigg),
		\end{align}
		where $\Theta=W, \gamma=1/2$ for the Wasserstein distance and $\Theta=K, \gamma=1/3$ in the case of the Kolmogorov distance. By virtue of \eqref{eq:sufficient2} we can bound the first term above as follows
		\begin{align*}
		&[\alpha_n(H^{\leq 1-\varepsilon}_n,H)]^{2\gamma}\wedge d_{\Theta}(\widetilde Z_{H^{\leq 1-\varepsilon}_n}^n,\,\mathcal N)\leq d_{\Theta}(\widetilde Z_{H^{\leq 1-\varepsilon}_n}^n,\,\mathcal N)\\ &\lesssim \({ \min\limits_{\substack{F\subseteq H^{\leq 1-\varepsilon}_n:\,e_F\geq1}}{P_F}\,n^{v_{F}}}\)^{-1/2}+\( \min\limits_{\substack{e\in E(H^{\leq 1-\varepsilon}_n):\, p_{|e|}>1/2}}(1-p_{|e|}){n^{|e|}}\)^{-1/2}\\
&\leq \({ \min\limits_{\substack{F\subseteq H^{\leq 1-\varepsilon}_n:\,e_F\geq1}}{P_F}\,n^{v_{F}}}\)^{-1/2}+\(\varepsilon  \min\limits_{\substack{e\in E(H^{\leq 1-\varepsilon}_n)}}{n^{|e|}}\)^{-1/2}\\
&\leq \frac2{\sqrt \varepsilon }\({ \min\limits_{\substack{F\subseteq H:\,e_F\geq1}}{P_F}\,n^{v_{F}}}\)^{-1/2}.
		\end{align*}
		Furthermore, by \eqref{eq:sufficient2} and \eqref{eq:estalpha}, for any $e\in E(H^{> 1-\varepsilon}_n)$ 
		we have 
		\begin{align*}
		&[\alpha_n(\{e\},H)]^{2\gamma}\wedge d_{K/W}(\widetilde Z_{\{e\}}^n,\,\mathcal N)\\
		&\lesssim \(\frac{n^{-|e|} (1-p_{|e|})}{\max_{F\subseteq H:\,e_F\geq1} n^{-v_F}Q_F/P_F}\)^{\gamma}\wedge \(\frac1{(1-p_{|e|})n^{|e|}}\)^{1/2}.
		\end{align*}
		Applying last two estimates to \eqref{eq:aux6} and extending the sum therein to all edges of $H$, we arrive at
		\begin{align*}
		d_{\Theta}(\widetilde{Z}_H^n,\,\mathcal N)&\lesssim \frac2{\sqrt \varepsilon }\({ \min\limits_{\substack{F\subseteq H:\,e_F\geq1}}{P_F}\,n^{v_{F}}}\)^{-1/2}+\varepsilon^{\gamma}\\
		&\ \ \ \ +\sum_{e\in E(H)}\(\frac{n^{-|e|} (1-p_{|e|})}{\max_{F\subseteq H:\,e_F\geq1} n^{-v_F}Q_F/P_F}\)^{\gamma}\wedge \(\frac1{(1-p_{|e|})n^{|e|}}\)^{1/2}.
		\end{align*}
		Taking $\varepsilon = \({ \min\limits_{\substack{F\subseteq H:\,e_F\geq1}}{P_F}\,n^{v_{F}}}\)^{-1/(2\gamma+1)} $ we obtain the first bounds form the assertion of the theorem.
	\end{proof}

	The rest of this section is devoted to comparison of the obtained bounds with the ones that follow from \cite{JLR} for the distance $d_1$. The approach in this book is a modification of the one from the article \cite{BKR}. Furthermore, a version for Kolmogorov distance has been established in \cite{ER}.

	\begin{definition}[{\cite[p. 11]{JLR}}]
		Let $\{X_i\}_{i\in \mathcal{I}}$ be a family of random variables on a common probability space. A dependency graph for $\{X_i\}_{i\in \mathcal{I}}$ is any graph $L$ with vertex set $V(L)=\mathcal{I}$ such that if $A$ and $B$ are two disjoint subsets of $\mathcal{I}$ with $$e_L(A, B) \defeq |\{\{a, b\}\in E(L) : a\in A, b\in B\}| = 0,$$ then
		the families $\{X_i\}_{i\in A}$ i $\{X_i\}_{i\in B}$ are mutually independent. In general, a dependency graph is not unique.
	\end{definition}
	
	Furthermore, by  $\bar{N}_{G}(v_1, \ldots, v_r)$ we denote the closed neighbourhood of vertices $v_1, \ldots, v_r\in V(G)$ in a simple graph $G$, i.e.
	$$\bar{N}_{G}(v_1, \ldots, v_r)=\big\{v\in V(G):\exists _{i\in[r]}\{v,v_i\}\in E(G)\big\}\cup\{v_1,\ldots,v_r\}.$$

%	Let  $\{H_i\}_{i\in \mathcal{I}}$ be a family of subhypergraphs of the complete hypergraph $K^*_N$ and, for any $i\in \mathcal I$, recall the notation introduced before i.e. $\I _{H_i} = \I _{\{H_i\subseteq \mathcal{H}(n, \boldsymbol{p})\}}$. Then, the  dependency graph for $\{\I _{H_i}\}_{i\in \mathcal{I}}$ has vertex set $\mathcal{I}$ and edge set $\{\{i,j\} : E(H_i)\cap E(H_j)\neq \emptyset\} $. Notice that sets of variables $\{\I _{H_i}\}_{i\in A}$ and $\{\I _{H_i}\}_{i\in B}$ for disjoint $A, B\subseteq \mathcal{I}$ are independent iff $e_L(A, B) = 0$\gs{, as required}\wm{w definicji grafu zaleznosci jest tylko implikacja w jedna strone, rownowaznosc jest czyms ekstra}. 
%	
%	

	\begin{theorem}[{\cite[Theorem 6.33]{JLR}}]\label{theo:chen_like_method}
		Suppose  $(S_n)_n$ is a sequence of random variables such that $S_n = \sum_{\alpha\in \mathcal{I}_n} X_{n, \alpha}$,  where for each $n$, $\{X_{n, \alpha}:\alpha\in \mathcal I_n\}$ is a family of random variables with dependency graph $L_n$. If there exist numbers $M_n$, $Q_n$ such that 
		$$\sum_{\alpha\in \mathcal{I}_n}\mathbb{E}|X_{n, \alpha}|\leq M_n,$$
		and for every $\alpha_1, \alpha_2 \in \mathcal{I}_n$,
		$$\sum_{\alpha\in \bar{N}_{L_n}(\alpha_1, \alpha_2)} \mathbb{E}\left[|X_{n, \alpha}|\big| X_{n, \alpha_1}, X_{n, \alpha_2}\right]\leq Q_n,$$
		then $$d_1(\widetilde{S}_n, \mathcal{N})\lesssim \frac{M_n Q_n^2}{(\mathrm{Var}[S_n])^{3/2}}.$$
	\end{theorem}
	
	Generalizing  the argument from Example 6.19 in \cite{JLR}, we prove the following bound.
	\begin{theorem}\label{theo:JKR_bound}
		For arbitrary an hypergraph $H$, we have
		$$d_1(\widetilde{Z}_H^n, \mathcal{N})\lesssim (1-P_H)\frac{\big(\min_{F\subseteq H, e_F \geq1} \frac{n^{v_F}P_F}{1 - P_F }\big)^{3/2}}{(\min_{F\subseteq H, e_F \geq 1}n^{v_F}P_F)^2}.$$
	\end{theorem}
	\begin{proof}
		The family of subhypergraphs $H_\alpha$ of $K_n^*$ which are isomorphic to $H$ will be denoted by $\{H_{\alpha}\}_{\alpha\in \mathcal{I}_n}$ with $\mathcal{I}_n=[N_H^n]$. %Let $I_{\alpha} = \I _{\{H_{\alpha}\subseteq \mathcal{H}(n, \boldsymbol{p})\}}$ and $X_{\alpha} = I_{\alpha}-\mathbb{E}I_{\alpha}$.
		For the family 
		$$\left\{Y_\alpha=\I _{\left\{ H_\alpha\subseteq \mathbb{H}(n, {p}) \right\}}-\E\[\I _{\left\{ H_\alpha \subseteq \mathbb{H}(n, {p}) \right\}}\]\right\}_{\alpha\in \mathcal I_n}$$
		we define a natural dependency graph $L_n$ as $(\mathcal I_n, E(L_n))$, where
		$$\{i,j\}\in E(L_n)\Leftrightarrow E(H_i)\cap E(H_j)\neq \emptyset.$$
		Notice that sets of variables $\{Y_\alpha\}_{\alpha\in A}$ and $\{Y_\alpha\}_{\alpha\in B}$ for disjoint $A, B\subseteq \mathcal{I}$ are independent if and only if  $e_L(A, B) = 0$, as they are built on disjoint sets of edges.
		Our intention is to apply Theorem \ref{theo:chen_like_method}. 
		First, note that
		\begin{equation}
			\begin{aligned}
				{\E}|Y_{{\alpha}}| %&= %\mathbb{E}[|\I_{H_\alpha}-\mathbb{E}\I_{H_\alpha}|\mid \I_{H_\alpha}=1]\mathbb{P}(\I_{H_\alpha}=1) + 
				%\mathbb{E}[|\I_{H_\alpha}-\mathbb{E}\I_{H_\alpha}|\mid \I_{H_\alpha}=1]\mathbb{P}(\I_{H_\alpha}=0)\\
				&=|1-P_H|P_H + 
				|0-P_H|(1-P_H)= 2 P_H (1-P_H),
			\end{aligned}
		\end{equation}
	and hence
		$$\sum_{\alpha\in \mathcal{I}_n}{\E}|Y_{{\alpha}}| = 2  P_H (1-P_H) |\mathcal{I}_n|\approx P_H (1-P_H) n^{v_H}.$$
		We therefore  set $M_n = P_H (1-P_H) n^{v_H}.$
		
		Next, fix $\alpha_1, \alpha_2 \in \mathcal{I}_n$ and observe that for all $\alpha\in \mathcal{I}_n$
		\begin{equation}
			\begin{aligned}
			{\E}\left[|Y_{{\alpha}}| \big| Y_{{\alpha_1}}, Y_{{\alpha_2}}\right]&\leq {\E}\left[\I_{\alpha} \big| \I_{H_{\alpha_1}}, \I_{H_{\alpha_2}}\right] + \E[\I_{H_\alpha}]\\
				%&=\mathbb{E}[\left(\I_{H_\alpha} \mid \I_{H_{\alpha_1}}=1, \I_{H_{\alpha_2}}=1\right)\I_{\{\I_{H_{\alpha_1}}=1, \I_{H_{\alpha_2}}\}}]\\
				%&\;\;+\mathbb{E}[\left(\I_{H_\alpha} \mid \I_{H_{\alpha_1}}=1, \I_{H_{\alpha_2}}=0\right)\I_{\{\I_{H_{\alpha_1}}=1, \I_{H_{\alpha_2}}=0\}}]\\
				%&\;\;+\mathbb{E}[\left(\I_{H_\alpha} \mid \I_{H_{\alpha_1}}=0, \I_{H_{\alpha_2}}=1\right)\I_{\{\I_{H_{\alpha_1}}=0, \I_{H_{\alpha_2}}=1\}}]\\
				%&\;\;+\mathbb{E}[\left(\I_{H_\alpha} \mid \I_{H_{\alpha_1}}=0, \I_{H_{\alpha_2}}=0\right)\I_{\{\I_{H_{\alpha_1}}=0, \I_{H_{\alpha_2}}=0\}}] + \mathbb{E}\I_{H_\alpha}\\
				&\leq 4 {\E}\left[\I_{H_\alpha} \mid \I_{H_{\alpha_1}}=1, \I_{H_{\alpha_2}}=1\right] + \E[\I_{H_\alpha}]\\
				&= 4 \mathbb{P}\(\I_{H_\alpha}=1 \mid \I_{H_{\alpha_1}}=1, \I_{H_{\alpha_2}}=1\) +P_{H_\alpha}\\
				&= 4 P_H/P_{H_\alpha\cap(H_{\alpha_1}\cup H_{\alpha_2})} + P_H\leq 5P_H/P_{H_\alpha\cap(H_{\alpha_1}\cup H_{\alpha_2})},
			\end{aligned}
		\end{equation}
		where $\I_{H_\alpha}=\I_{\{H_\alpha\subseteq \rh(n,  \p )\}}$.
	%The number of subhypergraphs of the hypergraph $F$ is bounded by $$\sum_{i=0}^{v_F}2^{2^{i}}{v_F \choose i} \leq 2^{2^{v_F}} 2^{v_F}\leq 2^{2^{v_H}} 2^{v_H},$$ in addition 
	For every subgraph $G\subseteq H_{\alpha_1}\cup H_{\alpha_2}$ there are $O( n^{v_H-v_F})$ choices of such $\alpha$ that $H_\alpha\cap(H_{\alpha_1}\cup H_{\alpha_2}) = G$.
	Since every hypergraph $H_\alpha\cap(H_{\alpha_1}\cup H_{\alpha_2})$ is isomorphic to some subhypergraph of  $H$, 
	\begin{equation}
		\begin{aligned}
			\sum_{\alpha\in \bar{N}_{L_n}(\alpha_1, \alpha_2)} \mathbb{E}\left(|Y_{{\alpha}}|\mid Y_{{\alpha_1}}, Y_{{\alpha_2}}\right)&\leq 5\sum_{\alpha\in \bar{N}_{L_n}(\alpha_1, \alpha_2)} P_H/P_{{H_\alpha\cap(H_{\alpha_1}\cup H_{\alpha_2})}}\\
			&\lesssim\max_{F\subseteq H, e_F\geq 1} \frac{n^{v_H}P_H}{n^{v_F}P_F} = \frac{n^{v_H}P_H}{\min_{F\subseteq H, e_F \geq 1}n^{v_F}P_F}.
		\end{aligned}
	\end{equation}
	Denoting the last expression by $Q_n$,  Theorem \ref{theo:chen_like_method} and Proposition \ref{theo:VarI_m} give us
	\begin{equation}
		\begin{aligned}
		d_1(\widetilde{Z}_H^n, \mathcal{N})&\lesssim \frac{M_n Q_n^2}{(\mathrm{Var}[Z_H^n])^{3/2}}\approx\frac{P_H (1-P_H) n^{v_H} \big(\frac{n^{v_H}P_H}{\min_{F\subseteq H, e_F \geq1}n^{v_F}P_F}\big)^2}{\big(P_H^2\max_{F\subseteq H, e_F \geq1} n^{2v_H-v_F} \frac{1 - P_F }{P_F}\big)^{3/2}}\\[5pt]
		&= (1-P_H)\frac{\big(\min_{F\subseteq H, e_F \geq1} \frac{n^{v_F}P_F}{1 - P_F }\big)^{3/2}}{(\min_{F\subseteq H, e_F \geq 1}n^{v_F}P_F)^2},
		\end{aligned}
	\end{equation}
	which completes the proof.
	\end{proof}

	The proof of Theorem \ref{theo:JKR_bound} is more straightforward than proofs of Theorem \ref{theo:sufficient} or Theorem \ref{theo:mainbounds}.
	Thus, a question may emerge whether that additional effort is justified.
	Below,  we exhibit an example that compares  those tools for normal approximation of $\widetilde{Z}_H^n$.

	Firstly, notice that  when all probabilities $p_r$ are separated from $1$, each of the mentioned Theorems give the same (believed to be optimal in the case of the model $\mathbb G(n,p)$) bound of the form as in \eqref{aux8}, which follows from 
	$$\min\limits_{\substack{e\in E(H):\, p_{|e|}>1/2}}(1-p_{|e|}){n^{|e|}} \gtrsim { \min\limits_{\substack{F\subseteq H:\,e_F\geq1}}{P_F}\,n^{v_{F}}},$$
	in the case of  \ref{theo:sufficient}, and from 
\begin{align*}
d_1(\widetilde{Z}_H^n, \mathcal{N})&\lesssim (1-P_H)\frac{\big(\min_{F\subseteq H, e_F \geq1} \frac{n^{v_F}P_F}{1 - P_F }\big)^{3/2}}{(\min_{F\subseteq H, e_F \geq 1}n^{v_F}P_F)^2}\lesssim 1  \frac{\big(\min_{F\subseteq H, e_F \geq1}{n^{v_F}P_F}\big)^{3/2}}{(\min_{F\subseteq H, e_F \geq 1}n^{v_F}P_F)^2}\\
&= \frac{1}{(\min_{F\subseteq H, e_F \geq 1}n^{v_F}P_F)^{1/2}},
\end{align*}
that is a consequence of Theorem \ref{theo:JKR_bound}.
	It leads to an observation that any potential differences between methods may arise only if at least one probabilities   $\{p_r\}$ approaches $1$. It suffices to consider a graph on $2$ vertices with one standard edge and one loop ($1$-edge).

	\begin{example}\label{example1}
	Let $H$ by a hypergraph with  $v_H = 2$ and two edges of sizes $1$ and $2$, whose existence probabilities are   $p_1 = 1-n^{-3}$ and $p_2 = 1-n^{-1}$.	
%	Clearly $$\max_{F\subseteq H, e_F \geq1} \frac{Q_F }{n^{v_F}P_F} \approx \frac{ n^{-1}}{n^{2}(1-n^{-1})} \approx n^{-3}.$$
	%Therefore, under the Proposition  \ref{theo:VarI_m}, we get
	%$$\mathrm{Var}Z_H^n\asymp n.$$
The first inequality of Theorem \ref{theo:sufficient} implies
	\begin{equation}
		\begin{aligned}
			d_{W/K}(\widetilde{Z}_H^n, \mathcal{N})\lesssim \(n^4\wedge n^3\)^{-1/2}+ \(n^4\wedge n^3\){\( \frac{n^{-3}}{n^{3}}\vee \frac{n^{-1}}{n^{6}}\)^{1/2}} =n^{-3/2}+1\approx 1.
		\end{aligned}
	\end{equation} 
 Furthermore, Theorem \ref{theo:mainbounds} gives us
	\begin{equation}
		\begin{aligned}
			&d_{W}(\widetilde{Z}_H^n, \mathcal{N})\lesssim (n\wedge n^2)^{-1/4}+\(\frac{ n^{-1}n^{-3}}{n^{-3}} \wedge \frac{1}{n^{-3}n}\)^{1/2} +
			\(\frac{ n^{-2}n^{-1}}{n^{-3}} \wedge \frac{1}{n^{-1}n^2}\)^{1/2}\approx n^{-1/4},\\[5pt]
			&d_{K}(\widetilde{Z}_H^n, \mathcal{N})\lesssim n^{-1/5}+\(\(n^{-1}\)^{1/3} \wedge \(n^2\)^{1/2}\) +
			\(\(1\)^{1/3} \wedge \(\frac{1}{n}\)^{1/2}\)\approx n^{-1/5}.
		\end{aligned}
	\end{equation} 
	On the other hand, Theorem \ref{theo:JKR_bound} yields the  bound 
		\begin{align*}
			d_{K}(\widetilde{Z}_H^n, \mathcal{N})&\lesssim(1-P_H)\frac{\big(\min_{F\subseteq H, e_F \geq1} \frac{n^{v_F}}{1 - P_F }\big)^{3/2}}{(\min_{F\subseteq H, e_F \geq 1}n^{v_F})^2}
			\approx n^{-1} \frac{\big(n^4\wedge n^3\big)^{3/2}}{n^2} =n^{3/2}.
		\end{align*}
One can see that Theorems \ref{theo:sufficient} and \ref{theo:JKR_bound} do not even ensure asymptotic normality. Theorem \ref{theo:mainbounds}, which arises by enhancement of Theorem \ref{theo:sufficient}, does. One can apply a similar procedure of enhancement to the bound from Theorem \ref{theo:JKR_bound}, however, one can see that the starting point is already much worse.
	
		\end{example}

\section{Quantitive fourth moment phenomenon}

Let $H_1,\ldots ,H_{N_H^n}$ be all the isomorphic copies of $H$ in the complete hypergraph $K_n^*$. Recalling the notations  $\I _{H_i}:=\I _{\left\{ H_i \subseteq \mathbb{H}(n, {p}) \right\}}$ and 
$$Y_i:= \I _{H_i}-\E\left[\I _{H_i}\right] $$ for any  $i\in\{1,\ldots ,N_H^n\}$, we get

$$Z_H^n-\E[Z_H^n]=\sum_iY_i.$$
 One of the main goals of this section is estimating the expression
\begin{equation}\label{x4}
	\sqrt{\left|\E\[\(\widetilde{Z}^n_H\)^4\]-3\right|}
%	=\sqrt{\left|\frac{\E\[(Z^n_H-\E Z_H^n)^4\]}{(\operatorname{Var} [Z_H^n])^2}-3\right|}
=\frac{\sqrt{\left|\E\[(Z_H^n-\E Z_H^n)^4\]-3(\operatorname{Var} [Z_H^n])^2\right|}}{\operatorname{Var} [Z_H^n]}.
\end{equation}
We start with deriving a suitable representation of the  last radicand.

For $H_1,\ldots ,H_m$ such that the graph $H_1\cap \(\bigcup_{i=2}^m H_i\)$ contains no edges  it holds that \begin{align}\label{eq:common}
\E\[ \prod_{i=1}^m Y_i\]=\E\[Y_{1}\]\E\[ \prod_{i=2}^mY_i\]=0,
\end{align} 
and such configurations will play a marginal role. We therefore introduce the following notation.
\begin{definition}\label{def:sep}We say that hypergraphs $H_1,\ldots ,H_m$, $m\geq2$, are {\it \Sep}  whenever there exists nonempty  $I\subseteq[m]$ such  that
$$E\(\Big(\bigcup_{i\in I} H_i\Big)\cap \Big(\bigcup_{i\in[m]\backslash I} H_i\Big)\)=\emptyset.$$
This will be denoted by $H_1,\ldots,H_m\sep$

On the other hand, if $H_1,\ldots ,H_m$ are not edgewise separable, we call them {\it \Ins} and  denote it by $H_1,\ldots ,H_m\ins$
\end{definition}
\begin{lemma}\label{lem:4mp-formula}It holds that%$Z_H^n$
\begin{align*}
&\E[(Z_H^n-\E Z_H^n)^4]-3(\operatorname{Var} [Z_H^n])^2=\sum_{i} \[\E [Y_i^4]-3\(\E[Y_i^2]\)^2\]\\
&\ \ \ +\sum_{\substack{i\neq j\\E(H_i\cap H_j)\neq\emptyset}}\[3\E [Y_i^2 Y_j^2]+4\E [Y_i^3 Y_j]-3\E [Y_i^2]\E [Y_j^2]-12 \E [Y_i^2] \E [Y_i Y_j]-6 (\E [Y_i Y_j])^2\]\\
	&\ \ \ + \sum_{\substack{|\{i,j,k\}|=3\\H_i,H_j,H_k\ins}}6\[\E[ Y_i^2 Y_j Y_k]-  \E [Y_i^2] \E [Y_j Y_k]- 2\E [Y_i Y_j]\E [Y_j Y_k]\]\\[6pt]
	&\ \ \ +\sum_{\substack{|\{i,j,k,l\}|=4\\H_i,H_j,H_k,H_l\ins}}\[\E [Y_i Y_j Y_k Y_l]
	-3 \E[ Y_i Y_j]\E [Y_k Y_l]\].
\end{align*}
\end{lemma}

\begin{proof}
 Let us rewrite the variance of $Z_H^n$ as follows 
\begin{align*}
	\left(\text{Var}Z_H^n\right)^2&=\left( \E\Big[\big(\sum_{i} Y_i \Big) ^2\Big]\right) ^2= \left( \sum_{i}\E[ Y_i^2]+\sum_{i\neq j}\E[Y_i Y_j]\right) ^2\nonumber\\
	&= \left(  \sum_{i} \E[Y_i^2]\right) ^2+2\left(  \sum_{i} \E[Y_i^2]\sum_{j\neq k}\E[Y_{j} Y_{k}]\right)+\left( \sum_{i\neq j}\E[Y_i Y_j]\right) ^2\\
	&=  \sum_{i} \(\E[Y_i^2]\)^2+ \sum_{i\neq j} \E[Y_i^2]\E[Y_j^2]
	+4 \sum_{i\neq j} \E[Y_i^2]\E[Y_i Y_j]
	+2 \sum_{|\{i,j,k\}|=3} \E[Y_i^2]\E[Y_j Y_k]\\
	&\ \ \ + 2\sum_{\substack{i\neq j}}\(\E[Y_i Y_j]\)^2
	+ 4\sum_{|\{i,j,k\}|=3}\E[Y_i Y_j]\E[Y_i Y_k]
	+ \sum_{|\{i,j,k,l\}|=4}\E[Y_i Y_j]\E[Y_k Y_l].
\end{align*}
Further, by the multinomial theorem we get
\begin{align*}
	\E\[\left( \sum_{H_i} Y_i\right)^4\]
	&=\E \[\sum_{\substack{k_1,\ldots ,k_{N_n}\geq 0 \\ k_1+\ldots +k_{N_n}=4}} \binom{4}{k_1,\ldots ,k_{N_n}}Y_{H_1}^{k_1}\cdot\ldots\cdot  Y_{H_{N_n}}^{k_{N_n}}\]\\
	&=\sum_{i}\E [Y_i^4]+3\sum_{i\neq j}\E[Y_i^2 Y_j^2]+4\sum_{i\neq j}\E[Y_i^3 Y_j]+6\sum_{|\{i,j,k\}|=3}\E [Y_i^2 Y_j Y_k]\nonumber\\
	&\ \ \ +\sum_{|\{i,j,k,l\}|=4}\E[Y_i Y_j Y_k Y_l] .
\end{align*}
Combining the last two representations, we arrive at
\begin{align*}
\E\[\(Z_H^n\)^4\]-3	\E\[\(Z_H^n\)^2\]^2&=S_1+S_2+S_3+S_4,
\end{align*}
where
\begin{align*}
S_1&=\sum_{i} \[\E [Y_i^4]-3\(\E[Y_i^2]\)^2\],\\
S_2&=3 \sum_{i\neq j}\Big[\E[ Y_i^2 Y_j^2]-\E [Y_i^2]\E [Y_j^2]\big]+\sum_{i\neq j}\Big[4\E[ Y_i^3 Y_j]-12  \E[Y_i^2]\E[Y_i Y_j]-6\(\E[Y_i Y_j]\)^2\Big],\\
	S_3&=6 \sum_{|\{i,j,k\}|=3}\Big[\E [Y_i^2 Y_j Y_k]-  \E [Y_i^2] \E [Y_j Y_k]\Big]-12\sum_{|\{i,j,k\}|=3}\E[Y_i Y_j]\E[Y_i Y_k],\\
	S_4&=\sum_{|\{i,j,k,l\}|=4}\Big[\E[ Y_i Y_j Y_k Y_l]
	-3 \E [Y_i Y_j]\E [Y_k Y_l]\Big].
\end{align*}
Our goal now is to reduce all summation to configurations of \Ins hypergraphs. We can clearly do this in both of the sums of $S_2$ and in the latter sum of $S_3$, which follows from the fact that for separable $H_i, H_j$ the variables  $Y_i$ and $Y_j$ are independent. 
We will therefore focus on the first sum of $S_3$ and the sum $S_4$ and show that they vanish when restricted TO \Sep configurations of hypergraphs. Regarding the one of $S_3$, let $H_i,H_j,H_k$ be \Sep and  consider two cases. If $H_i$ has no common edge with $H_j\cup H_k$, then 
$$\E [Y_i^2 Y_j Y_k]=\E [Y_i^2]\E[ Y_j Y_k].$$
Another possibility is that  $H_j$ or $H_k$ has no common edge with {the} other two  hypergraphs. Assuming, without loss of the generality,  it is $H_k$, we  get
\begin{align*}
\E [Y_i^2 Y_j Y_k]&=\E [Y_i^2 Y_j]\E[ Y_k]=0,\\
\E [Y_i^2] \E [Y_j Y_k]&=\E [Y_i^2] \E [Y_j]\E[ Y_k]=0.
\end{align*}
Thus, every term for separable  $H_i,H_j,H_k$ vanishes.

Eventually, we turn our attention to the sum $S_4$. Assume  $H_i,H_j,H_k,H_l$ are edgewise separable. If one of the hypergraphs has no common edges with any other, then $\E [Y_i Y_j Y_k Y_l]=0$. 
The only other possible case is when  one pair of graphs has no common edge with  another pair. Furthermore, for a fixed quadruple of such graphs, the expected value $\E[ Y_i Y_j Y_k Y_l]$ appears $24$ times in the sum, as one may simply permute the indices, while in the sum with the  $\E[ Y_i Y_j]\E[ Y_k Y_l]$ every term appears only $8$ times, since one element in each pair determines the other one (in fact, every third term in this sum is non-zero). Consequently, we have
$$\sum_{\substack{|\{i,j,k,l\}|=4\\H_i,H_j,H_k,H_l\sep}}\E [Y_i Y_j Y_k Y_l]=\frac{24}8\sum_{\substack{|\{i,j,k,l\}|=4\\H_i,H_j,H_k,H_l\sep}}\E[ Y_i Y_j]\E[ Y_k Y_l],$$
as required.
\end{proof}
Next, we will estimate the terms in Lemma \ref{lem:4mp-formula}.
\begin{lemma}\label{EYI}
	For integer $r\geq2$ and \Ins $H_{i_1},\ldots, H_{i_r}$ there exists $c_1=c_1(r)\in (0, 1)$ such that for $p<c_1$ it holds that
	\begin{align}\label{EYY1}
		\E [Y_{{i_1}}  \cdot\ldots\cdot Y_{i_r}] \stackrel{r}{\approx} p^{e_{H_{i_1}\cup\ldots H_{i_r}}}.
	\end{align}
	Furthermore, for $p\in(0, 1)$ we have 
		\begin{align}\label{EYY2}
		   \E \left|Y_{{i_1}}  \cdot\ldots\cdot Y_{i_r}\right|\leq2e_H^r(1-p),
	\end{align}
	and
	\begin{align}\label{EYY3}
	\E [Y_{i_1}  \cdot\ldots\cdot Y_{i_r}]\geq (-1)^{r}\P\big(H_{i_1},\ldots,H_{i_r}\not\in \mathbb H(n,p)\big)-c_2(1-p)^2.
	\end{align}
	for some $c_2=c_2(r)>0$.
	
	\end{lemma}

\begin{proof}Let us rewrite 
\begin{align*}
\E [Y_{i_1}  \cdot\ldots\cdot Y_{i_r}]&=\E\[ \prod_{k=1}^r\(\I_{H_{i_k}}-p^{e_H}\)\]\\
&=\E \[\I_{H_{i_1}} \cdot\ldots\cdot \I_{H_{i_r}}+ \prod_{A\varsubsetneq[r]}(-1)^{r-|A|}\I _{\left\{\bigcup_{l\in A}H_{i_l}\right\}}p^{(r-|A|)e_H}\]\\
&=p^{e_{H_{i_1}\cup\ldots H_{i_r}}}\[1+ \prod_{A\varsubsetneq[r]}(-1)^{r-|A|}p^{\left\{e_{\bigcup_{l\in A}H_{i_l}}+\sum_{l\in [r]\backslash A}e_{H_{i_l}}-e_{H_{i_1}\cup\ldots H_{i_r}}\right\}}\]\\
&=p^{e_{H_{i_1}\cup\ldots H_{i_r}}}\big(1+2^rO(p)\big),
\end{align*}
where the last equality follows from the assumption of edgewise inseparability of $H_{i_1},\ldots,H_{i_r}$, which implies that the hypergraphs $\bigcup_{l\in A}H_{i_l}$ and $\bigcup_{l\in[r]\backslash A}H_{i_l}$ share a common edge and hence
$$e_{\bigcup_{l\in A}H_{i_l}}+\sum_{l\in [r]\backslash A}e_{H_{i_l}}>e_{H_{i_1}\cup\ldots H_{i_r}}.$$ Taking $p$ sufficiently close to $0$, we get the estimate \eqref{EYY1}.

We pass to the next assertion of the lemma.
For every $i\in[N_H^n]$ and $r\in\N$ it holds
\begin{align*}
\E[(Y_i)^r]&=(1-p^{e_H})^rp^{e_{H}}+\left(-p^{e_H}\right)^r(1-p^{e_{H}})\\
&\leq [e_H(1-p)]^rp^{e_{H}}+\left(p^{e_H}\right)^re_H(1-p)\leq 2e_H^r(1-p).
\end{align*}
Hence,  the generalized H\"older inequality gives us
\begin{align*}
\E\[Y_{i_1}  \cdot\ldots\cdot Y_{i_r}\]&\leq \(\E\[Y_{i_1}^{r}\]  \cdot\ldots\cdot \E\[Y_{i_r}^{r}\]\)^{1/r}\leq  2e_H^r(1-p),
\end{align*}
which proves us  the inequality \eqref{EYY2}. Furthermore, since $\E\[|\I_{H_i}-1|^k\]=1-p^{e_H}$ for any $k\in\N$ and $i\in[N_H^n]$, the generalized H\"older inequality applied once again lets us bound
%\begin{align*}
%\E \[\big(\I_{H_i}-1\big) \big(\I_{H_j}-1\big)\]&=\E[\I_{H_i\cup H_j}]-\E[\I_{H_i}]-\E[\I_{H_j}]+1\\
%&=1-p^{e_{H_i}}-p^{e_{H_j}}\(1-p^{e_{H_i\cup H_j}-e_{H_j}}\)\leq2e_H (1-p),
%\end{align*}
%we obtain
\begin{align*}
&\E\[Y_{i_1}  \cdot\ldots\cdot Y_{i_r}\]\\
&=\E \[\big((\I_{H_{i_1}}-1)+(1-p^{e_H})\big) \cdot\ldots\cdot \big((\I_{H_{i_r}}-1)+(1-p^{e_H})\big)\]\\
&=\E \[(\I_{H_{i_1}}-1) \cdot\ldots\cdot(\I_{H_{i_r}}-1)\]+O_{r}\bigg((1-p^{e_H})^2+(1-p^{e_H})\E \big[|\I_{H_{1}}-1|^{r-1}\big]\bigg)\\
&=(-1)^r\P\bigg(H_{i_1}, \ldots, H_{i_r}\notin\mathbb{H}(n, p) \bigg)+O_{r}((1-p)^2),
\end{align*}
where $\left|O_{r}((1-p)^2)\right|\leq c(1-p)^2$ for some $c=c(r)$.
\end{proof}

The next two theorems deal with $p$ very close to $0$ or very close to $1$.

\begin{theorem}\label{thm:4mpp->0}
There exists $\varepsilon=\varepsilon(H)>0$  such that for $p<\varepsilon$   we have
$$ {\E\[\widetilde Z_H^4\]-3}\approx\frac1{ \min\limits_{\substack{F\subseteq H\\e_F\geq1}}n^{v_F}p^{e_F}}.$$
\end{theorem}
\begin{proof} Due to \eqref{eq:homoVar}, it suffices to show that 
\begin{align*}
\E\[\(Z_H^n-\E[Z_H^n]\)^4\]-3\big(\mathrm{Var}[Z_H^n]\big)^2\approx n^{4 v_H} p^{4 e_H}\Bigg( \max_{\substack{F\subseteq H\\e_F\geq1} } n^{-v_F} p^{-e_F}\Bigg)^3.
\end{align*}
Let us rewrite the formula from Lemma \ref{lem:4mp-formula} in the following manner
\begin{align*}
&\E\[\(Z_H^n-\E[Z_H^n]\)^4\]-3\big(\mathrm{Var}[Z_H^n]\big)^2=\sum_{i} \E [Y_i^4]\[1-3\frac{\(\E[Y_i^2]\)^2}{\E[ Y_i^4]}\]\\
&+\sum_{\substack{i\neq j\\E(H_i\cap H_j)\neq\emptyset}}\hspace{-15pt}\(3\E [Y_i^2 Y_j^2]+4\E[ Y_i^3 Y_j]\)\[1-\frac{3\E [Y_i^2]\E [Y_j^2]+12 \E [Y_i^2] \E [Y_i Y_j]+6 (\E [Y_i Y_j])^2}{3\E [Y_i^2 Y_j^2]+4\E [Y_i^3 Y_j]}\]\\
	&+6 \sum_{\substack{|\{i,j,k\}|=3\\H_i, H_j, H_k \ins}}\E[ Y_i^2 Y_j Y_k]\[1- \frac{ \E [Y_i^2] \E [Y_j Y_k]+2 \E [Y_i Y_j]\E[ Y_j Y_k]}{\E[ Y_i^2 Y_j Y_k]}\]\\[6pt]
	&+\sum_{\substack{|\{i,j,k,l\}|=4\\H_i, H_j, H_k, H_l \ins}}\E [Y_i Y_j Y_k Y_l]
	\[1-3 \frac{\E[ Y_i Y_j]\E [Y_k Y_l]}{\E [Y_i Y_j Y_k Y_l]}\].
\end{align*}
We will show that the expressions in brackets are comparable with a constant for $p$ sufficiently small. Indeed, by \eqref{EYY1} and the equality $\E\[\I_{H_i}\]=p^{e_{H}}$, $i\in[N_n]$, we have 
\begin{align*}
\frac{\(\E[Y_i^2]\)^2}{\E[ Y_i^4]}\approx\frac{\(\E[\I_{H_i}]\)^2}{\E[ \I_{H_i}]}=p^{e_H},\ \ \ p<c,
\end{align*}
for some $c>0$. Next, using the equivalence  $\I_{H_i}\I_ {H_j}=\I_{H_i\cup H_j}$ (which may seem peculiar when compared to the algebra of classical indicators) and the estimate \eqref{EYY1}, we get 
\begin{align*}
&\frac{3\E [Y_i^2]\E [Y_j^2]+12 \E [Y_i^2] \E [Y_i Y_j]+6 (\E [Y_i Y_j])^2}{3\E[ Y_i^2 Y_j^2]+4\E [Y_i^3 Y_j]}\\
&\approx \frac{\E [\I_{H_i}]\E [\I_{H_j}]+ \E [\I_{H_i}] \E [\I_{H_i} \I_{H_j}]+ (\E [\I_{H_i} \I_{H_j}])^2}{\E [\I_{H_i} \I_{H_j}]}\\
&=\frac{\E [\I_{H_i}]\E [\I_{H_j}]+ \E [\I_{H_i}] \E [\I_{H_i\cup H_j} ]+ (\E [\I_{H_i\cup H_j} ])^2}{\E[ \I_{H_i\cup H_j}]}\\
&=\frac{p^{2e_H}+ p^{e_H}p^{e_{H_i\cup H_j}}+ p^{2e_{H_i\cup H_j}} }{p^{e_{H_i\cup H_j}}}=p^{2e_H-{e_{H_i\cup H_j}}}+p^{e_H}+p^{e_{H_i\cup H_j}}\leq 3p,
\end{align*}
where the last inequality follows from the edgewise inseparability of $H_i$ and $H_j$, which ensures that $1\leq e_{H_i\cup H_j}<2e_H$. 
%Similarly, we also get 
%\begin{align*}
%&\frac{ \E [Y_i^2] \E [Y_j Y_k]+2 \E Y_i Y_j\E Y_j Y_k}{\E Y_i^2 Y_j Y_k}\approx \frac{ \E [\I_{H_i}] \E [\I_{H_j\cup H_k}]+2 \E \I_{H_i\cup H_j}\E \I_{H_j\cup H_k}}{\E \I_{H_i\cup H_j\cup H_k}}\\
%&=p^{e_{H_i}+e_{H_j\cup H_k}-e_{H_i\cup H_j\cup H_k}}+p^{e_{H_i\cup H_j}+e_{H_j\cup H_k}-e_{H_i\cup H_j\cup H_k}}\\
%&=p^{e_{H_i\cap\(H_j\cup H_k\)}}+p^{e_{H_i\cup(H_j\cap H_k)}}\leq 2p,
%\end{align*}
%as well as 
%\begin{align*}
%&\frac{\E Y_i Y_j\E Y_k Y_l}{\E Y_i Y_j Y_k Y_l}\approx \frac{\E \I_{H_i} \I_{H_j}\E \I_{H_k} \I_{H_l}}{\E \I_{H_i} \I_{H_j} \I_{H_k} \I_{H_l}}=\frac{p^{H_i\cup H_j}p^{H_k\cup H_l}}{p^{H_i\cup H_j\cup H_k\cup H_l}}=p^{\(H_i\cup H_j\)\cap\(H_k\cup H_l\)}\leq p.
%\end{align*}
We proceed analogously with the last two sums, and eventually arrive at
\begin{align*}
&\E\[\(Z_H^n\)^4\]-3	\E\[\(Z_H^n\)^2\]^2\\[6pt]
&\approx \sum_{i} \E [\I_{H_i}]+\hspace{-10pt}\sum_{\substack{i\neq j\\E(H_i\cap H_j)\neq\emptyset}}\hspace{-10pt}\E [\I_{H_i}\I_{H_j}]+\hspace{-10pt} \sum_{\substack{|\{i,j,k\}|=3\\H_i, H_j, H_k \ins}}\hspace{-10pt}\E [\I_{H_i} \I_{H_j} \I_{H_k}]+\hspace{-10pt}\sum_{\substack{|\{i,j,k,l\}|=4\\H_i, H_j, H_k, H_l \ins}}\hspace{-10pt}\E [\I_{H_i} \I_{H_j} \I_{H_k} \I_{H_l}]\\[6pt]
&=\sum_{\substack{i,j,k,l\\H_i, H_j, H_k, H_l \ins}}\E [\I_{H_i} \I_{H_j} \I_{H_k} \I_{H_l}].
\end{align*}
What has left is to prove that
\begin{align}\label{eq:1111}
\sum_{\substack{i,j,k,l\\H_i, H_j, H_k, H_l \ins}}\E [\I_{H_i} \I_{H_j} \I_{H_k} \I_{H_l}]\approx n^{4 v_H} p^{4 e_H}\left( \max_{\substack{F\subseteq H\\e_F\geq1} } n^{-v_F} p^{-e_F}\right)^3.
\end{align}
We start with the upper bound. Fix $i,j,k,l\in[N_H^n]$ such that  $H_i,H_j,H_k, H_l$ are edgewise inseparable. Without loss of generality,  one can assume that each of the collections $H_{i}, H_{j}$ and $H_{i}, H_{j}, H_{k}$ and $ H_{i}, H_{j}, H_{k}, H_{l}$ is \Ins as well, and therefore each of the hypergraphs $A_1=H_{i}\cap H_{j}$, $A_2=\(H_{i}\cup H_{j}\)\cap H_{k}$  and $A_3=\(H_{i}\cup H_{j}\cup H_{k}\)\cap H_{l}$ is isomorphic  to a subhypergraph of $H$ with at least one edge. Additionally, it holds that
$$v_{H_i\cup H_j\cup H_k\cup H_l}=4v_H-v_{A_1}-v_{A_2}-v_{A_3}.$$
Thus, we obtain
\begin{align*}
&\E [\I_{H_i} \I_{H_j} \I_{H_k} \I_{H_l}]=p^{e_{H_i\cup H_j\cup H_k\cup H_l}}=p^{e_{H_{i}}}p^{e_{H_j}-e_{A_1}}p^{e_{H_{k}}-e_{A_2}}p^{e_{H_{l}}-e_{A_3}}\\[5pt]
&=p^{4e_H}\(n^{-v_{A_1}}p^{-e_{A_1}}\)\(n^{-v_{A_2}}p^{-e_{A_2}}\)\(n^{-v_{A_3}}p^{-e_{A_3}}\)n^{v_{A_1}+v_{A_2}+v_{A_3}}\\[5pt]
&=n^{-v_{H_i\cup H_j\cup H_k\cup H_l}}n^{4v_H}p^{4e_H}\(n^{-v_{A_1}}p^{-e_{A_1}}\)\(n^{-v_{A_2}}p^{-e_{A_2}}\)\(n^{-v_{A_3}}p^{-e_{A_3}}\)\\[5pt]
%&=n^{-v_{H_i\cup H_j\cup H_k\cup H_l}}\(n^{v_{H_{i}}}p^{e_{H_{i}}}\)\(n^{v_{H_{j}}-v_{A_1}}p^{e_{H_{j}}-e_{A_1}}\)\(n^{v_{H_{k}}-v_{A_2}}p^{e_{H_{k}}-e_{A_2}}\)\(n^{v_{H_{l}}-e_{A_3}}p^{v_{H_{l}}-e_{A_3}}\)\\
&\leq n^{-v_{H_i\cup H_j\cup H_k\cup H_l}}n^{4v_H}p^{4e_H}\left( \max_{\substack{F\subseteq H\\e_F\geq1} } n^{-v_F} p^{-e_F}\right)^3,
\end{align*}
and consequently
\begin{align*}
\sum_{\substack{i,j,k,l\\H_i, H_j, H_k, H_l \ins}}\E [\I_{H_i} \I_{H_j} \I_{H_k} \I_{H_l}]
&\leq n^{4v_H}p^{4e_H}\left( \max_{\substack{F\subseteq H\\e_F\geq1} } n^{-v_F} p^{-e_F}\right)^3\sum_{\substack{i,j,k,l\\H_i, H_j, H_k, H_l \ins}}n^{-v_{H_i\cup H_j\cup H_k\cup H_l}}.
\end{align*}
Since $v_H\leq v_{H_i\cup H_j\cup H_k\cup H_l}=4v_H-v_{A_1}-v_{A_2}-v_{A_3}\leq 4v_H-3$, we estimate the last sum as follows 
\begin{align*}
\sum_{\substack{i,j,k,l\\H_i, H_j, H_k, H_l \ins}}n^{-v_{H_i\cup H_j\cup H_k\cup H_l}}&=\sum_{s=v_H}^{4v_H-3}\sum_{\substack{i,j,k,l\\H_i, H_j, H_k, H_l \ins\\v_{H_1\cup H_j\cup H_k\cup H_l}=s}}n^{-s}
\leq \sum_{s=v_H}^{4v_H-3}4!\sum_{\substack{F\subseteq K_n^*\\v_{F}=s}}n^{-s}\approx 1,
\end{align*}
which gives the  upper bound in \eqref{eq:1111}. 

Let $H^*\subseteq H$ realize the maximum in \eqref{eq:1111}, i.e., it holds that $\max_{\substack{F\subseteq H\\e_F\geq1} } n^{-v_F} p^{-e_F}=n^{-v_{H^*}} p^{-e_{H^*}}$. In order to obtain the lower estimate in \eqref{eq:1111}, we narrow the sum on the left-hand side only to the quadruples $(i,j,k,l)$ from the set 
$$I:=\left\{\{i_1,i_2,i_3,i_4\}: \exists _{H'\subseteq K_n^*}\(H'\simeq H^* \wedge \forall _{r,s\in\{1,2,3,4\}, r\neq s} H_{i_{r}}\cap H_{i_s}=H'\)\right\}. $$
In other words, the graphs $H_i,H_j, H_k, H_l$ have one common part, which is isomorphic to $H^*$, and any two of them have no other common parts. Note that $|I|\geq {n\choose 4v_H-3v_{H^*}}$ as  on every $4v_H-3v_{H^*}$ vertices one can generate at least one and unique  configuration of hypergraphs corresponding to indices from $I$. Thus we get 
\begin{align*}
\sum_{\substack{i,j,k,l\\H_i, H_j, H_k, H_l \ins}}\E[ \I_{H_i} \I_{H_j} \I_{H_k} \I_{H_l}]&\geq \sum_{\{i,j,k,l\}\in I}\E [\I_{H_i} \I_{H_j} \I_{H_k} \I_{H_l}]=|I|p^{4e_H-3e_{H^*}}\\
&\gtrsim n^{4v_H-3v_{H^*}}p^{4e_H-3e_{H^*}} =n^{4v_H}p^{4e_H}\(\max_{\substack{F\subseteq H\\e_F\geq1} } n^{-v_F} p^{-e_F}\)^3,
\end{align*}
which ends the proof.
\end{proof}

\begin{theorem}\label{thm:4mpp->1}
There exists $\varepsilon=\varepsilon(H)>0$  such that for $p>1-\varepsilon$  we have
$$ {\E\[\widetilde Z_H^4\]-3}\approx\frac1{(1-p)n^{\min\{|e|:e\in E(H)\}}}.$$
\end{theorem}

\begin{proof}
Applying  \eqref{EYY2} to Lemma \ref{lem:4mp-formula} we get for $p\in(0,1)$

	\begin{align}
	&\left|\E\big[(Z_H^n-\E Z_H^n)^4\big]-3(\operatorname{Var} [Z_H^n])^2\right|\\
		&\lesssim (1-p)\bigg[\sum_{i}1
		 + \sum_{\substack{|\{i,j\}|=2\\ {E\(H_i\cap H_j\)\neq\emptyset}}} 1+\sum_{\substack{|\{i,j,k\}|=3\\ {H_i, H_j, H_k \ins}}}1+\sum_{\substack{|\{i,j,k,l\}|=4\\ {H_i, H_j, H_k, H_l \ins}}}1\bigg]\\[6pt]
		 &\lesssim (1-p)\left[n^{v_H}+n^{2v_H-\min\{|e|:e\in E(H)\}}+n^{3v_H-2\min\{|e|:e\in E(H)\}}+n^{4v_H-3\min\{|e|:e\in E(H)\}}\right]\\\label{aux3}
		 &\leq 4(1-p)n^{4v_H-3\min\{|e|: e\in E(H)\}}.
	\end{align}
	Similarly, applying both of the inequalities  \eqref{EYY2} and \eqref{EYY3} to Lemma \ref{lem:4mp-formula}, we may write
	\begin{align*}
	&\E(Z_H^n-\E Z_H^n)^4-3(\operatorname{Var} [Z_H^n])^2\\
		&\geq \sum_{\substack{|\{i,j,k,l\}|=4\\ H_i, H_j, H_k, H_l \ins}}\P\big(H_{i},H_j,H_k,H_{l}\not\in \mathbb H(n,p)\big)+O\big((1-p)^2\big)n^{4v_H-3{\min\{|e|:e\in E(H)\}}}.
	\end{align*}
	Next, we bound the above sum as follows
	\begin{align*}
	&\hspace{-15pt}\sum_{\substack{|\{i,j,k,l\}|=4\\ H_i, H_j, H_k, H_l \ins}}\P\big(H_{i},H_j,H_k,H_{l}\not\in \mathbb H(n,p)\big)\\
	&\geq \sum_{\substack{|\{i,j,k,l\}|=4\\|E({H_i\cap H_j\cap H_k\cap H_l})|=1}}\P\big(H_{i},H_j,H_k,H_{l}\not\in \mathbb H(n,p)\big)\\
	&\geq \sum_{\substack{|\{i,j,k,l\}|=4\\|E({H_i\cap H_j\cap H_k\cap H_l})|=1}}\P\big(H_i\cap H_j\cap H_k\cap H_l\not\in \mathbb H(n,p)\big)\\
	&=  \sum_{\substack{|\{i,j,k,l\}|=4\\|E({H_i\cap H_j\cap H_k\cap H_l})|=1}}(1-p)\approx (1-p)n^{4v_H-3{\min\{|e|:e\in E(H)\}}},
	\end{align*}
	where the last estimate follows from the fact that on every set of  $4v_H-3{\min\{|e|:e\in E(H)\}}$  vertices one can build at least one and unique configuration of four graphs that are isomorphic to $H$ and their only common part is an edge of the possibly smallest  size.  This gives us
		\begin{align*}
	&\E\big[(Z_H^n-\E Z_H^n)^4\big]-3(\operatorname{Var} [Z_H^n])^2
\geq (1-p)n^{4v_H-3{\min\{|e|:e\in E(H)\}}}\(1+O\big((1-p)\big)\).
	\end{align*}
	Combining this with \eqref{aux3} and \eqref{eq:homoVar}, we obtain the assertion of the theorem. 
\end{proof}

\begin{proof}[\textbf{Proof of Theorem \ref{thm:4mp}}] From \eqref{eq:sufficient2} we get
\begin{align*}
d_{W/K}\(\widetilde{Z}_H^n,\mathcal N\)&\lesssim\({ \min\limits_{\substack{F\subseteq H:\,e_F\geq1}}p^{e_F}\,n^{v_{F}}}\)^{-1/2}+\( \min\limits_{\substack{e\in E(H)}}(1-p){n^{|e|}}\)^{-1/2}\\
&\leq \({ \min\limits_{\substack{F\subseteq H:\,e_F\geq1}}(1-p)p^{e_F}\,n^{v_{F}}}\)^{-1/2},
\end{align*}
which is the first estimate in the assertion. We turn our attention to the other one.
In view of Theorems \ref{thm:4mpp->0} and \ref{thm:4mpp->1} it is enough to show that for any $\delta\in(0,1/2)$ \eqref{eq:4mp1} holds whenever $p\in(\delta, 1-\delta)$ (with constants possibly depending on $\delta$). For such $p$ we have
$$\({{(1-p) \min\limits_{\substack{F\subseteq H:\,e_F\geq1}}{p}^{e_F}\,n^{v_{F}}}}\)^{-1}\stackrel{\delta}{\approx }\({{ \min\limits_{\substack{F\subseteq H:\,e_F\geq1}}\,n^{v_{F}}}}\)^{-1}={n^{-\min\{|e|:e\in E(H)\}}},$$
it is therefore sufficient to show that 
\begin{align*}
{\left|\E\[\(\widetilde{Z}_H^n\)^4 \]-3\right|}\stackrel{\delta}{\lesssim }{n^{-\min\{|e|:e\in E(H)\}}}.
\end{align*}
Indeed, it is a consequence of \eqref{aux3} and \eqref{eq:homoVar}.
		\end{proof}

		\section{Necessary and sufficient conditions}
		
%CZY TEN LEMAT JEST W OGOLE POTRZEBNY? CHYBA NIE
%	\begin{lemma}\label{lem:necessary}
%	If there exists $c\in (0,1)$ such that  $p_{|e|}\leq c$  holds for all $e\in E(H)$, then
%		 $$\widetilde{Z}_H^n \overset{d}{\longrightarrow} \mathcal{N}(0, 1)\Longleftrightarrow  \min\limits _{F\subseteq H:\, e_F\geq1}R_Gn^{v_F}\longrightarrow \infty.$$
%	\end{lemma}
%	\begin{proof}
%	$(\Leftarrow)$ This implication follows from OSZACOWANIA ODLEGLOSCI.\\
%	$(\Rightarrow)$ We adapt arguments from the proof \ldots  . By \ldots , there is a constant $C>0$ such that
%	\begin{align*}
%\widetilde{Z}_H^n=\frac{{Z}_H^n-\E[{Z}_H^n]}{{\mathrm{Var}}[{Z}_H^n]}\geq -\frac{\E[{Z}_H^n]}{{\mathrm{Var}}[{Z}_H^n]}\geq -C\sqrt{ \min\limits _{F\subseteq H:\, e_F\geq1}\frac{R_G\,n^{v_F}}{Q_F}}.
%\end{align*}
%Since the distribution $\mathcal{N}(0, 1)$ is supported on the whole real line, the last expression has to tend to $-\infty$. Due to the assumption $p_{|e|}\leq c$, $e\in E(H)$, for every $F\subseteq H$ it holds $Q_F\approx 1$ and therefore $ \min\limits _{F\subseteq H:\, e_F\geq1}R_Gn^{v_F}\rightarrow \infty$.
%	\end{proof}

We start this section with the case when all edge probabilities are bounded away from zero. 

	\begin{proposition}\label{theo:lind_cond}
		If  $p_{|e|}\geq 1/2$  holds for all $e\in E(H)$, then
		$\widetilde{Z}_H^n\overset{d}{\longrightarrow} \mathcal{N}(0, 1)$ if and only if for every $e\in E(H)$ it holds that
		$$\frac1{(1-p_{|e|})n^{|e|}} \wedge \frac{n^{-|e|}(1-p_{|e|})}{\max_{e'\in E(H)} n^{-|e'|}(1-p_{|e'|})}  \stackrel{n\rightarrow \infty}{\longrightarrow} 0.$$
	\end{proposition}
	\begin{proof}
	We start with showing that   ${\mathrm{Var} [I_m]}\rightarrow 0$ for $m\geq 2$ (cf. \eqref{eq:Z=II}). Indeed, by Theorem 	 \ref{theo:VarI_m} and the observation that  $P_H\approx 1$, for such $m$ it holds that
	\begin{align*}
		\mathrm{Var} [I_m] &\approx  \frac{n^{2v_H}}{\mathrm{Var} Z_H^n}
	\max_{F\subseteq H, e_F=m}\frac{Q_F}{n^{v_F}}\\
	&\leq  \frac{n^{2v_H}}{\mathrm{Var} Z_H^n}
	\max_{F\subseteq H, e_F=m}\frac1n\max _{e\in E(F)}\frac{(1-p_{|e|})}{n^{|e|}}\\
	&\lesssim \frac1n\mathrm{Var} I_1\leq \frac1n,
		\end{align*}
	where the first inequality follows simply from the fact that any hypergraph with at least two edges contains an edge with {fewer} vertices than the whole hypergraph. Consequently, by Chebyshev's inequality, $\widetilde{Z}_H^n-I_1$ converges in probability to zero, and therefore Slutsky's theorem gives us the equivalence
	\begin{align*}
	\widetilde{Z}_H^n\overset{d}{\longrightarrow} \mathcal{N}(0, 1) \ \ \Longleftrightarrow\ \  I_1\overset{d}{\longrightarrow} \mathcal{N}(0, 1).
	\end{align*}
Since 
$$I_1=\sum_{i=1}^{2^n-1}a_i\widetilde X_i,
	$$
	where 
	$$a_i=a_i(n)=\frac{P_H}{\sqrt{\mathrm{Var} [Z_H^n]}} \sqrt{\frac{1-p_{|e_i|}}{p_{|e_i|}}} \sum_{A\subseteq \N\backslash\{i\}}\I _{A\cup \{i\} \sim H},$$
 is a sum {of} independent random variables, we may employ some classical tools. First, we verify the Feller-L\'evy condition. For any $i$ associated with an edge of size that appears in $H$ (for other ones $a_i$ equals zero), we have
 \begin{align}\label{eq:a_i}
 a_i\approx \sqrt{\frac{1-p_{|e_i|}}{{\mathrm{Var} [Z_H^n]}}}\,n^{v_H-|e_i|},
\end{align}
 and therefore 
 \begin{align*}
 \(\mathrm{Var} [I_1]\)^{-1}\mathrm{Var}\[a_i \widetilde X_i\]
 &\approx \(\frac{n^{2v_H}}{\mathrm{Var} [Z_H^n]}\max_{e\in E(H)}\frac{1-p_{|e|}}{n^{|e|}}\)^{-1}
 \frac{1-p_{|e_i|}}{{\mathrm{Var} [Z_H^n]}}\(n^{v_H-|e_i|}\)^2\\
  &= \[\(\max_{e\in E(H)}\frac{1-p_{|e|}}{n^{|e|}}\)^{-1}
 \frac{1-p_{|e_i|}}{{n^{|e_i|}}}\]\frac1{n^{|e_i|}}\leq \frac1{n^{|e_i|}}\leq \frac1n.
  \end{align*}
  Thus, the maximum, with respect to $i$, of the above expressions  tends to zero. This means that the Lindeberg condition 
 \begin{align}\label{eq:Lind}
\forall _{ \varepsilon >0}\lim_{n\rightarrow\infty}\frac1{\mathrm{Var}[I_1]}\sum_{i=1}^{2^{n}-1}\mathbb{E}\[\(a_i \widetilde X_i\)^2\I _{\{|a_i \widetilde X_i|> \varepsilon \sqrt{\mathrm{Var}[I_1]}\}}\]=0
 \end{align} 
  is both: sufficient and necessary for $I_1$, and hence for $\widetilde Z_H^n$ as well,  to be asymptotically normal.  Using the fact that $X_i\in\{0,1\}$, we rewrite the above expected value as follows 
  \begin{align}\nonumber
  \mathbb{E}\[\(a_i \widetilde X_i\)^2\I _{\{|a_i \widetilde X_i|> \varepsilon \sqrt{\mathrm{Var}[I_1]}\}}\]&=
    \mathbb{E}\[\I _{\{X_i=0\}} \frac{a^2_ip_{|e_i|}}{1-p_{|e_i|}}\I _{\left\{|a_i| \sqrt{p_{|e_i|}/(1-p_{|e_i|})}> \varepsilon \sqrt{\mathrm{Var}[I_1]}\right\}}\]\\\nonumber
    &\ \ \ +\mathbb{E}\[\I _{\{X_i=1\}} \frac{a^2_i(1-p_{|e_i|})}{p_{|e_i|}}\I _{\left\{|a_i| \sqrt{(1-p_{|e_i|})/p_{|e_i|}}> \varepsilon \sqrt{\mathrm{Var}[I_1]}\right\}}\]\\[6pt]\nonumber
    &=
   {a^2_ip_{|e_i|}}\I _{\left\{|a_i| \sqrt{p_{|e_i|}/(1-p_{|e_i|})}> \varepsilon \sqrt{\mathrm{Var}[I_1]}\right\}}\\\nonumber
    &\ \ \ +{a^2_i(1-p_{|e_i|})}\I _{\left\{|a_i| \sqrt{(1-p_{|e_i|})/p_{|e_i|}}> \varepsilon \sqrt{\mathrm{Var}[I_1]}\right\}}\\\label{eq:aux1}
    &\approx {a^2_i}\I _{\left\{|a_i| \sqrt{p_{|e_i|}/(1-p_{|e_i|})}> \varepsilon \sqrt{\mathrm{Var}[I_1]}\right\}},
  \end{align}
where the last estimate comes from the inequality  $p_{|e_i|}/(1-p_{|e_i|})\geq (1-p_{|e_i|})/p_{|e_i|}$, which is valid for $p_{|e_i|}\geq1/2$. Furthermore, 
  since the limit in  \eqref{eq:Lind}  is supposed to be true for any $\varepsilon>0$, we can equivalently replace $a_i=a_i(n)$ in the indicator with any other sequence $b_i(n)$ such that $a_i(n)\approx b_i(n)$. Consequently, by \eqref{eq:a_i},\eqref{eq:aux1} and the estimate $\mathrm{Var}[I_1]\approx 1$, the condition \eqref{eq:Lind} is equivalent to
  \begin{align*}
\forall _{ \varepsilon >0}\lim_{n\rightarrow\infty}\sum_{i\in \mathcal E_n}{\frac{1-p_{|e_i|}}{{\mathrm{Var} [Z_H^n]}}}\,n^{2v_H-2|e_i|}\I _{\{n^{v_H-|e_i|}\,>\, \varepsilon \sqrt{\mathrm{Var}[Z_H^n]}\}} = 0,
  \end{align*}
 where $\mathcal E_n=\{i\in [2^n-1]: \exists_{e\in E(H)}|e_i|=|e|  \}$. Observing that in $K_n^*$ there are ${n\choose |e|}\approx n^{|e|}$ edges of the same size as any given $e\in E(H)$, we  divide the above sum into parts associated with sizes of edges and get
  $$\sum_{i\in \mathcal E_n}{\frac{1-p_{|e_i|}}{{\mathrm{Var} [Z_H^n]}}}\,n^{2v_H-2|e_i|}\I _{\{n^{v_H-|e_i|}\,>\, \varepsilon \sqrt{\mathrm{Var}[Z_H^n]}\}}
  \approx \sum_{e\in E(H)}{\frac{1-p_{|e|}}{{\mathrm{Var} [Z_H^n]}}}\,n^{2v_H-|e|}\I _{\{n^{v_H-|e|}\,>\, \varepsilon \sqrt{\mathrm{Var}[Z_H^n]}\}}.$$
  Let us note that the sums above may not be equal,  since there may be several edges in $H$ of the same size. Denoting  additionally $t_n=t_n(e)={\frac{1-p_{|e|}}{{\mathrm{Var} [Z_H^n]}}}\,n^{2v_H-|e|}$,  we reformulate further the Lindeberg condition into
   \begin{align}\nonumber
&\forall _{ \varepsilon >0}\lim_{n\rightarrow\infty}\sum_{e\in E(H)}t_n\I _{\{t_n\,>\,\varepsilon^2 (1-p_{|e|})n^{|e|}\}}=0\\\nonumber
&\Leftrightarrow\forall _{ e\in E(H)}\forall _{ \varepsilon >0}\lim_{n\rightarrow\infty}t_n\I _{\{t_n\,>\,\varepsilon^2 (1-p_{|e|})n^{|e|}\}}=0\\\label{eq:aux2}
&\Leftrightarrow\forall _{ e\in E(H)}\lim_{n\rightarrow\infty}\(t_n\wedge\frac1{(1-p_{|e|})n^{|e|}}\)=0.
  \end{align}
  Let us explain the last equivalence. By \eqref{eq:VarZ} we obtain
  \begin{equation}\label{eq:t_n}
  	t_n\approx \frac{ (1-p_{|e|})n^{-|e|}}{\max\limits_{e'\in E(H)}(1-p_{|e'|})n^{-|e'|}}
  \end{equation}
  and the implication $\Leftarrow$  follows from the fact that $t_n$ has bounded order of magnitude.
  To deduce the converse one, let us observe that if there was an increasing sequence of natural numbers $({n_k})_{k\geq1}$ and $e\in E(H)$ such that $t_{n_k}\wedge \frac1{(1-p_{|e|})n_k^{|e|}}> \delta$ for some $\delta > 0$ and all $k\geq1$, we would get for $\varepsilon =\delta$ 
  $$t_{n_k}\I _{\{t_{n_k}\,>\,\varepsilon^2 (1-p_{|e|}){n_k}^{|e|}\}}=t_{n_k}\I _{\big\{t_{n_k}/\((1-p_{|e|}){n_k}^{|e|}\)\,>\,\delta^2 \big\}}=t_{n_k}\not\rightarrow0.$$ 
  Finally, \eqref{eq:t_n} applied to \eqref{eq:aux2} gives us the assertion of the proposition.
  \end{proof}

%	\begin{theorem}
%		Let $H$  be a hypergraphs without isolated vertices. 
%	Then $\widetilde{Z}_H^n \overset{d}{\rightarrow} \mathcal{N}(0,1)$
%		holds if and only if
%		
%		
%$$ \min\limits _{F\subseteq H:\, e_F\geq1}P_Fn^{v_F}\rightarrow \infty$$
%		 and 
%		\begin{equation}\label{long_necessary}
%		\frac1{n^{|e|}(1-p_{|e|})} \wedge \frac{n^{-|e|} (1-p_{|e|})}{\max_{F\subseteq H:\,e_F\geq1} n^{-v_F}Q_F/P_F} \overset{n\rightarrow \infty}{\longrightarrow} 0 \;\;\;\forall_{e\in E(H)}.
%		\end{equation}
%	\end{theorem}
	\begin{proof}[\textbf{Proof of Theorem \ref{thm:conditions}}]
		The sufficiency  of the conditions for the asymptotic normality of $\widetilde{Z}_H^n$ follows from Theorem \ref{theo:mainbounds}. We will therefore focus on the necessity of the conditions. 

		($\Longrightarrow$) 
		Assume $\widetilde{Z}_H^n \overset{d}{\rightarrow} \mathcal{N}$. Analogously as in the proof of the implication  {\it 1) $\Rightarrow$ 2)} of Theorem \ref{thm:homoequiv} presented in Section \ref{sec:homo}, by the estimates \eqref{eq:Eestgen} and \eqref{eq:Vestgen} we get
		\begin{align}\label{eq:aux5}
		\frac{\E[Z_H^n]}{\sqrt{\mathrm{Var}[Z_H^n]}}\approx \sqrt{ \min\limits _{F\subseteq H:\, e_F\geq1}\frac{P_Fn^{v_F}}{Q_F}}\ \longrightarrow \infty.
		\end{align}
		If $F\subseteq H$ contains only edges $e$ such that $p_{|e|}>1/2$, $e\in E(F)$, then 
		$$P_Fn^{v_F}\geq \frac {n^{v_F}}{2^{e_F}}\geq \frac{n}{2^{e_H}}.$$ 
Consider now  $F\subseteq H$ with at least one edge such that $p_{|e|}\leq 1/2$ and recall that $H_n^{\leq1/2}\subseteq H$	stands for the hypergraph consisting of those edges  $e\in E(H)$ for which $p_{|e|}\leq1/2$ ({cf.} \eqref{eq:H^}). For such a hypergraph $F$ we have
\begin{align*}
{P_Fn^{v_F}}&\geq \frac{P_{F\cap H_n^{\leq1/2}}}{2^{e_H-e_{F\cap H_n^{\leq1/2}}}}n^{v_{F\cap H_n^{\leq1/2}}}\frac1{Q_{F\cap H_n^{\leq1/2}}2^{e_{F\cap H_n^{\leq1/2}}}}=2^{-e_H}\frac{P_{F\cap H_n^{\leq1/2}}n^{v_{F\cap H_n^{\leq1/2}}}}{Q_{F\cap H_n^{\leq1/2}}}\\
&\geq 2^{-e_H}\min\limits _{F'\subseteq H:\, e_{F'}\geq1}\frac{P_{F'}n^{v_{F'}}}{Q_{F'}}.
\end{align*}
 Consequently 
  $$\min\limits _{F\subseteq H:\, e_{F}\geq1}P_{F}n^{v_{F}}\gtrsim n\wedge\(\min\limits _{F\subseteq H:\, e_{F}\geq1}\frac{P_{F}n^{v_{F}}}{Q_{F}}\),$$
	which, in view of \eqref{eq:aux5},  gives us the first condition in the assertion. The rest of the proof is devoted to deriving the other one, i.e., the condition \eqref{long_necessary}.

		 We will show that any sequence of positive integers diverging to infinity contains a subsequence for which the limit \eqref{long_necessary} holds. 
		 First, for any such a sequence we can choose  a subsequence $(n_k)_{k\geq1}$ such that for every $e\in E(H)$ the probabilities $p_{|e|}=p_{|e|}(n_k)$ are convergent as $k\rightarrow\infty$. Depending on the values of the limits, we decompose the graph $H$ as
		$$H=\bar H^{<1}\cup \bar H^1, $$
		where the hypergraphs $\bar H^{<1}$ and $\bar H^1$ share no common edge (i.e., $e_{\bar H^{<1}\cap \bar H^1}=0$), contain no isolated vertices and
		
		\begin{itemize}
			\item $\lim_{k\rightarrow\infty}p_{|e|}(n_k) < 1$ whenever $e\in \bar H^{<1}$,
			\item $\lim_{k\rightarrow\infty}p_{|e|}(n_k) = 1$ whenever $e\in \bar H^{1}$.
		\end{itemize}
	 Additionally, let us denote by $\mathcal H$ {the} family of nontrivial subhypergraphs of $H$ without isolated vertices, that are  subhypergraphs neither of $\bar H^{<1}$ nor of $\bar H^{1}$, 
		\begin{align*}
		\mathcal H:=\{F\subseteq H: F\not\subseteq \bar H^{<1}, \ F\not\subseteq \bar H^1, \ e_F\geq1 \}.
		\end{align*}
		This and Proposition \ref{prop:subHoeff} allow us to represent $\widetilde{Z}_H^n$ as follows 
		\begin{align}\nonumber
		\widetilde{Z}_H^n&=\sum_{F\subseteq \bar H^{<1}}I_F+\sum_{F\subseteq \bar H^{1}}I_F+\sum_{F\in \mathcal H}I_F\\
		&\label{eq:decompZ2}
		=\alpha^{<1}_n\widetilde Z_{\bar H^{<1}}^n
		+\alpha^{1}_n\widetilde Z_{\bar H^{1}}^n
		+\sum_{F\in \mathcal H}I_F,
		\end{align}
		where
		\begin{align*}
		\alpha^{<1}_n=C_{\bar H^{<1}}{n- v_{\bar H^{<1}}\choose v_H-v_{\bar H^{<1}}}\frac{P_{H}\sqrt{\mathrm{Var} [Z_{\bar H^{<1}}^n]}}{P_{\bar H^{<1}}\sqrt{\mathrm{Var} [Z_{H}^n]}},\hspace{40pt}
		\alpha^{1}_n=C_{\bar H^{1}}{n- v_{\bar H^{1}}\choose v_H-v_{\bar H^{1}}}\frac{P_{H}\sqrt{\mathrm{Var} [Z_{\bar H^{1}}^n]}}{P_{\bar H^{1}}\sqrt{\mathrm{Var} [Z_{H}^n]}}.
		\end{align*}
		Since for any $F\in \mathcal H$ the graphs $F\cap \bar H^{<1}$ and $F\cap H^{1}$ contain at least one edge,  it holds that 
		\begin{align*}
		\frac{Q_F}{P_Fn^{v_F}}&=	\frac{Q_{F\cap \bar H^{<1}}}{P_{F\cap \bar H^{<1}}\,n^{v_{F\cap \bar H^{<1}}}}\frac{Q_{F\cap \bar H^{1}}}{P_{F\cap \bar H^{1}}\, n^{v_F- v_{F\cap \bar H^{<1}}}}
		\leq \(\max _{K\subseteq H:\, e_K\geq1}\frac{Q_K}{P_Kn^{v_K} }\)\(\max_{K\subseteq \bar H^1,\,e_K\geq1}\frac{Q_K}{P_K}\),
		\end{align*}
		and hence, by Theorem \ref{theo:VarI_m}, we obtain the following bound for the variance of the last sum in \eqref{eq:decompZ2}
		\begin{align*}
		\mathrm{Var}\[\sum_{F\in \mathcal H}I_F\]&=\sum_{F\in \mathcal H}\mathrm{Var}\[I_F\]\approx \frac{P_H^2n_k^{2v_H}}{\mathrm{Var} [Z_H^n]}\sum_{F\in \mathcal H}
			\frac{Q_F}{P_Fn_k^{v_F}}\lesssim \max_{K\subseteq H^1,\,e_K\geq1}\frac{Q_K}{P_K}\stackrel{k\rightarrow\infty}{\longrightarrow} 0,
		\end{align*}
		where the last limit follows from the fact that for any $K\subseteq \bar H^1$ one has $Q_K\rightarrow0$ and $P_K\rightarrow1$. Thus,  the term $\sum_{F\in \mathcal H}I_F$ has no influence on asymptotic normality of $\widetilde{Z}_H^{n_k}$. We therefore  conclude 
		\begin{align}\label{eq:Z+Z}
		\alpha^{<1}_{n_k}\widetilde Z_{\bar H^{<1}}^{n_k}
		+\alpha^{1}_{n_k}\widetilde Z_{\bar H^{1}}^{n_k}\stackrel{d}{\longrightarrow}\mathcal N.
		\end{align}
	Observe now that since $\{I_F\}_{F\subseteq H}$ are uncorrelated and $\Var{\widetilde{Z}_H^n}=1$, then \eqref{eq:decompZ2} implies that $0\leq \alpha^{<1}_{n_k}, \alpha^{1}_{n_k}\leq 1$. This ensures existence of a subsequence $n_{k_l}$ such that 
$$\lim_{l\rightarrow \infty}\alpha^{<1}_{n_{k_l}}=\alpha^{<1},\hspace{40pt}
\lim_{l\rightarrow \infty}\alpha^{1}_{n_{k_l}}=\alpha^{1},$$
for some constants $0\leq \alpha^{<1}, \alpha^{1}\leq1$.

We are now prepared to deduce the condition \eqref{long_necessary}. It is clearly satisfied for $e\in E(\bar H^{<1})$ since then $n_k^{|e|}(1-p_{|e|})\rightarrow\infty$ as the expression in the parentheses is bounded away from $0$. Consider  now any $e\in E(\bar H^1)$.  If $\alpha^1=0$, then \eqref{long_necessary} is  satisfied due to \eqref{eq:estalpha}. If $\alpha^1\neq0$, we take advantage of   the already proven condition  $ \min\limits _{F\subseteq H:\, e_F\geq1}P_Fn_k^{v_F}\longrightarrow \infty$, but  applied to $\bar H^{<1}$. Since all $p_{|e|}$, $e\in E(\bar H^{<1})$, are bounded away from $1$ for almost all $n_k$, the aforementioned condition is by Theorem \ref{theo:sufficient} a  sufficient condition for asymptotic normality of $\widetilde Z_{\bar H^{<1}}^{n_k}$.  
Combining this with \eqref{eq:Z+Z} and independence $\widetilde Z_{\bar H^{<1}}^{n_k}$  of  $\widetilde Z_{\bar H^{1}}^n$, we conclude $\widetilde Z_{\bar H^{1}}^{n_k}\stackrel{d}{\rightarrow}\mathcal N$. Then, Proposition \ref{theo:lind_cond} together with \eqref{eq:Vestgen} give us
\begin{align*}
\frac{n_k^{-|e|} (1-p_{|e|})}{\max_{F\subseteq H:\,e_F\geq1} n_k^{-v_F}Q_F/P_F}&=\frac{n_k^{-|e|} (1-p_{|e|})}{\max_{e'\in E(H^1)} n_k^{-|e'|}(1-p_{|e'|})}\frac{\max_{e'\in E(\bar H^1)} n_k^{-|e'|}(1-p_{|e'|})}{\max_{F\subseteq H:\,e_F\geq1} n_k^{-v_F}Q_F/P_F}\\
&\approx\frac{n_k^{-|e|} (1-p_{|e|})}{\max_{e'\in E(\bar H^1)} n_k^{-|e'|}(1-p_{|e'|})}(\alpha_{n_k}^1)^2\longrightarrow\infty.
\end{align*}
 This ends the proof.
	\end{proof}

		\begin{proof}[\textbf{Proof of Theorem \ref{thm:homoequiv}}]\ \\
		{\it 1) $\Rightarrow$ 2).} Already proven just after formulation of the theorem.\\
		{\it 2) $\Rightarrow$ 3).} Since 
		\begin{align*}(1-p)\min_{\substack{F\subseteq H:\,e_F\geq1}}p^{e_F}n^{v_F}&\approx\left\{
		\begin{array}{ll}
		 \min_{\substack{F\subseteq H:\,e_F\geq1}}p^{e_F}n^{v_F},&\text{ for }p\leq\frac12,\\[7pt]
		 (1-p)n^{\min\{|e|:e\in E(H)\}},&\text{ for }p>\frac12,
		\end{array} 
	\right.
	\end{align*} 
	we have
		$$(1-p)\min_{\substack{F\subseteq H:\,e_F\geq1}}p^{e_F}n^{v_F}\geq \(\min_{\substack{F\subseteq H:\,e_F\geq1}}p^{e_F}n^{v_F}\)\wedge\((1-p)n^{\min\{|e|:e\in E(H)\}}\),$$
		which tends to infinity in view of {\it 2)}.\\
		Implications {\it 3) $\Rightarrow$ 4)} and {\it 4) $\Rightarrow$ 1)} follow from Theorem \ref{thm:4mp}, which is proven in the previous section.
		\end{proof}

\section*{Acknowledgements}	
	Wojciech Michalczuk and Grzegorz Serafin were supported  by the National Science Centre, Poland, grant no. 2021/43/D/ST1/03244.

\section{Appendix}
Surprisingly, the literature on bounds on distances between sums of random variables seems to be very limited. In this section we derive some inequalities that are crucial for the proofs of the main results of the article.

\begin{proposition}\label{prop:variance}
For any random variables $X$ and $Y$ such that $\E[Y]=0$ and $\E[Y^2]<\infty$, we have
\begin{align}
d_W(X+Y,\,\mathcal N)&\leq d_W\( X,\,\mathcal N\)+(\mathrm{Var}[Y])^{1/2},\\\label{eq:1/3}
d_K(X+Y,\,\mathcal N)&\leq d_K\( X,\,\mathcal N\)+\tfrac43(\mathrm{Var}[Y])^{1/3}.
\end{align}
\end{proposition}
\begin{proof}Regarding  the Wasserstein distance, we have
\begin{align}
d_W(X+Y,\,\mathcal N)&\leq d_W(X+Y,X)+d_W(X,\,\mathcal N)\\
&\leq \sup_{\text{Lip}(h)\leq1}\E\[\left|h(X+Y)-h(X)\right|\]+d_W(X,\,\mathcal N)\\
&\leq \E\[\left|Y\right|\]+d_W(X,\,\mathcal N)\leq (\mathrm{Var}[Y])^{1/2}+d_W(X,\,\mathcal N),
\end{align}
as required. Let us turn our attention to the bound for the Kolmogorov distance. 
 For any $\varepsilon>0$ we have
\begin{align*}
\big|\P(X+Y \leq t)-\P(\mathcal N \leq t)\big|&\leq \big|\P(X+Y \leq t, |Y|\leq \varepsilon)-\P(\mathcal N \leq t)\big|+\P(|Y|>\varepsilon).
\end{align*}
Since
\begin{align}
&\P(X+Y\leq t, |Y|\leq \varepsilon)-\P(\mathcal N\leq t)\\
&\leq  \P(X-\varepsilon \leq t, |Y|\leq \varepsilon)-\P(\mathcal N \leq t)\\
&\leq  \P(X-\varepsilon \leq t)-\P(\mathcal N \leq t)\\
&=  \big(\P(X \leq t+\varepsilon)-\P(\mathcal N \leq t+\varepsilon)\big)+\big(\P(\mathcal N \leq t+\varepsilon)-\P(\mathcal N \leq t)\big)\\
&\leq  d_K\(X,\,\mathcal N\)+\frac{\varepsilon}{\sqrt{2 \pi}},
\end{align}
and, similarly, 
\begin{align}
&\P(X+Y \leq t, |Y|\leq \varepsilon)-\P(\mathcal N \leq t)\\
&\geq  \P(X+\varepsilon  \leq t, |Y|\leq \varepsilon)-\P(\mathcal N \leq t)\\
&\geq  \P(X+\varepsilon  \leq t)-\P\(|Y|> \varepsilon\)-\P(\mathcal N \leq t)\\
&=  \big(\P(X \leq t-\varepsilon)-\P(\mathcal N \leq t-\varepsilon)\big)-\P\(|Y|> \varepsilon\)-\big(\P(\mathcal N \leq t)-\P(\mathcal N \leq t-\varepsilon)\big)\\
&\geq  -d_K\(X,\,\mathcal N\)-\P\(|Y|> \varepsilon\)-\frac{\varepsilon}{\sqrt{2 \pi}},
\end{align}
we have
\begin{align*}
d_K(X+Y,\,\mathcal N)
&\leq d_K(X,\,\mathcal N)+2\P\(|Y|> \varepsilon\)+\frac{\varepsilon}{\sqrt{2 \pi}}\\
&\leq d_K(X,\,\mathcal N)+2\frac{\mathrm{Var}[Y]}{\varepsilon^2}+\frac{\varepsilon}{\sqrt{2 \pi}},
\end{align*}
were we applied Chebyshev's inequality in the last inequality. Due to arbitrariness of $\varepsilon$ and the fact $\min\{\frac a{\varepsilon^2}+\varepsilon:\varepsilon>0\}=3(a/4)^{1/3}$, $a>0$, we get
\begin{align*}
d_K(X+Y,\,\mathcal N)
&\leq d_K(X,\,\mathcal N)+\frac{3\(\mathrm{Var}[Y]\)^{1/3}}{\(4  \pi\)^{1/3}}\leq d_K(X,\,\mathcal N)+\tfrac43\(\mathrm{Var}[Y]\)^{1/3},
\end{align*}
which ends the proof.
\end{proof}

\begin{remark} Let us show that the power $\tfrac13$ in \eqref{eq:1/3}  is optimal. Namely, for 
 $X\sim \mathcal N(0,1)$ and  $\varepsilon>0$  define 
$$Y=-\varepsilon\mathbf1_{\{X\in(0,\varepsilon)\}}+\varepsilon\mathbf1_{\{X\in(-\varepsilon-\delta,-\varepsilon)\}},$$
where $\delta>0$ is chosen such that  $\P\(X\in(0, \varepsilon)\)=\P\(X\in(-\varepsilon-\delta,-\varepsilon)\)$.
Then $\E\[Y\]=0$ and \\  $$\text{Var}[Y]=2  \varepsilon ^2\P\(X\in(0, \varepsilon)\)\leq \sqrt{\frac{2}{\pi}}\varepsilon^3.$$
Noting additionally that $d_K(X,\,\mathcal N)=0$, we conclude 
\begin{align*}
d_K\(X+Y, \mathcal N\)&\geq |\P\(X+Y\leq 0\)-\P\(\mathcal N\leq0\)|=\P\(X]leq\varepsilon\)-\P\(\mathcal N\leq0\)\\
&=\P\(X\in(0, \varepsilon)\) \geq \varepsilon \frac{e^{-\varepsilon^2/2}}{\sqrt{2\pi}}
\geq  \frac{e^{-\varepsilon^2/2}}{{(4\pi)^{1/3}}}\(\text{Var}[Y]\)^{1/3}.
\end{align*}
Since additionally $ d_K(X,\,\mathcal N)=0$, one can see that $\(\text{Var}[Y]\)^{1/3}$ is of optimal order for $\text{Var}[Y]$ close to zero.
\end{remark}

\begin{cor}\label{cor:variance}
For uncorrelated random variables $X$ and $Y$ such that $\E[X]=\E[Y]=0$ and $\mathrm{Var}[X]+\mathrm{Var}[Y]=1$, we have
\begin{align}
d_W(X+Y,\,\mathcal N)&\leq d_W\( \frac{X}{\sqrt{\mathrm{Var}[X]}},\,\mathcal N\)+2(\mathrm{Var}[Y])^{1/2},\\\label{eq:1/3cor}
d_K(X+Y,\,\mathcal N)&\leq d_K\( \frac{X}{\sqrt{\mathrm{Var}[X]}},\,\mathcal N\)+3(\mathrm{Var}[Y])^{1/3}.
\end{align}
\end{cor}
\begin{proof}
We present only the proof for Kolmogorov distance, as the one for the Wasserstein distance is analogous and even slightly simpler.  From the previous proposition we get
\begin{align}
d_K(X,\,\mathcal N)&=d_K\(\frac{X}{\sqrt{\mathrm{Var}[X]}}-\(1-\sqrt{\mathrm{Var}[X]}\)\frac{X}{\sqrt{\mathrm{Var}[X]}},\,\mathcal N\)\\
&\leq d_K\( \frac{X}{\sqrt{\mathrm{Var}[X]}},\,\mathcal N\)+\frac43\(1-\sqrt{\mathrm{Var}[X]}\)^{2/3}\\
&= d_K\( \frac{X}{\sqrt{\mathrm{Var}[X]}},\,\mathcal N\)+\frac43\(\frac{{\mathrm{Var}[Y]}}{1+\sqrt{\mathrm{Var}[X]}}\)^{2/3}\\
&\leq d_K\( \frac{X}{\sqrt{\mathrm{Var}[X]}},\,\mathcal N\)+\frac43\({{\mathrm{Var}[Y]}}\)^{1/3},
\end{align}
where the last inequality follows from the assumption $\mathrm{Var}[Y]=1-\mathrm{Var}[X]\leq1$. Applying this to \eqref{eq:1/3} we obtain the assertion.
\end{proof}

\begin{proposition}\label{prop:independent}
Let $(X_i)_{i=1}^n$ be a collection of centred  independent random variables satisfying  $\sum_{i=1}^n\mathrm{Var}[X_i]=1$. Then 
\begin{align*}
d_W\(\sum_{i=1}^nX_i,\,\mathcal N\)&\leq2\sum_{i=1}^n  \[\({\mathrm{Var}[X_i]}\)^{1/2}\wedge d_W\(\frac {X_i}{\sqrt{\mathrm{Var}[X_i]}},\,\mathcal N\)\],\\
d_K\(\sum_{i=1}^nX_i,\,\mathcal N\)&\leq3\sum_{i=1}^n  \[\({\mathrm{Var}[X_i]}\)^{1/3}\wedge d_K\(\frac {X_i}{\sqrt{\mathrm{Var}[X_i]}},\,\mathcal N\)\].
\end{align*}
\end{proposition}
\begin{proof}First, let us observe that it suffices to assume that all $X_i$'s are non-degenerate. \\
Consider the case $n=2$.
For a random variable $X$, let us denote its distribution by $\mu_X$. Furthermore, let $\mathcal A$ be a family of Borel functions  $h:\R\rightarrow\R$ that  is closed for operations $h(\cdot)\rightarrow h(a\,(\cdot))$ and $h(\cdot)\rightarrow h((\cdot))+b$ for any $a\in(0,1), b\in\R$. Then, for $\mathcal N_{1}\sim \mathcal N(0,{\Var[X_i]})$, $i=1,2$, independent  of each other and of $X_1, X_2$, we have
\begin{align*}
&\sup_{h\in \mathcal A}\left|\E\[h(X_1+X_2)-h(\mathcal N)\]\right|=\sup_{h\in \mathcal A}\left|\E\[h(X_1+X_2)-h(\mathcal N_1+\mathcal N_2)\]\right|\\
&\leq \sup_{h\in \mathcal A}\left|\E\[h(X_1+X_2)-h( X_1+\mathcal N_2)\]\right|+\sup_{h\in \mathcal A}\left|\E\[h(X_1+\mathcal N_2)-h(\mathcal N_1+\mathcal N_2)\]\right|\\
&\leq \E\bigg[\sup_{h\in \mathcal A}\left|\E\Big[h(X_1+X_2)-h( X_1+\mathcal N_2)\big|X_1\Big]\right|\bigg]\\
&\ \ \ +\E\bigg[\sup_{h\in \mathcal A}\left|\E\Big[h(X_1+\mathcal N_2)-h( N_1+\mathcal N_2)\big|\mathcal N_2\Big]\right|\bigg]\\
&\leq \sum_{i=1}^2\sup_{h\in \mathcal A}\left|\E\[h\({X_i}\)-h\(\mathcal N_i\)\]\right|\\
&= \sum_{i=1}^2\sup_{h\in \mathcal A}\left|\E\[h\(\sqrt{\mathrm{Var}[{X_i}]}\frac{X_i}{\sqrt{\mathrm{Var}[{X_i}]}}\)-h\(\sqrt{\mathrm{Var}[{X_i}]}\,\mathcal N\)\]\right|\\
&\leq  \sum_{i=1}^2\sup_{h\in \mathcal A}\left|\E\[h\(\frac{X_i}{\sqrt{\mathrm{Var}[{X_i}]}}\)-h\(\mathcal N\)\]\right|.
\end{align*}
Taking $\{h\in\R^{\R}:\text{Lip(1)}\leq1\}$ or $\{(-\infty,t):t\in\R\}$ as $\mathcal A$, we arrive at
\begin{align}
d_W(X_1+X_2, \mathcal N)&\leq d_W\(\frac{X_1}{\sqrt{\mathrm{Var}[{X_1}]}},\, \mathcal N\)+d_W\(\frac{X_2}{\sqrt{\mathrm{Var}[{X_2}]}},\, \mathcal N\),\\\label{aux4}
d_K(X_1+X_2, \mathcal N)&\leq d_K\(\frac{X_1}{\sqrt{\mathrm{Var}[{X_1}]}},\, \mathcal N\)+d_K\(\frac{X_2}{\sqrt{\mathrm{Var}[{X_2}]}},\, \mathcal N\).
\end{align}
From this point the proofs for both of the distances are analogous, we therefore focus only  on the Kolmogorov distance. Combining \eqref{aux4} with \eqref{eq:1/3cor} we get 
\begin{align}\label{aux7}
d_K(X_1+X_2, \mathcal N)&\leq d_K\(\frac{X_1}{\sqrt{\mathrm{Var}[{X_1}]}},\, \mathcal N\)+3\[\({\mathrm{Var}[X_2]}\)^{1/3}\wedge d_K\(\frac{X_2}{\sqrt{\mathrm{Var}[{X_2}]}},\, \mathcal N\)\].
\end{align}
Since we can clearly exchange $X_1$ with $X_2$, we complete the proof in this case ($n=2$) with 
$$d_K(X_1+X_2, \mathcal N)\leq1=\mathrm{Var}[{X_1}]+\mathrm{Var}[{X_2}]\leq 3\[\(\mathrm{Var}[{X_1}]\)^{1/3}+\(\mathrm{Var}[{X_2}]\)^{1/3}\].$$ 
In the case $n\geq3$ it is enough to iterate the bound \eqref{aux7} until obtaining 
\begin{align*}
d_K\(\sum_{i=1}^n X_i, \mathcal N\)&\leq d_K\(\frac{X_1+X_2}{\sqrt{\mathrm{Var}[{X_1+X_2}]}},\, \mathcal N\)+3\sum_{i=3}^n\[\({\mathrm{Var}[X_i]}\)^{1/3}\wedge d_K\(\frac{X_i}{\sqrt{\mathrm{Var}[{X_i}]}},\, \mathcal N\)\],
\end{align*}
and then apply the already proven bound for $n=2$ to the first term on the right-hand side.
\end{proof}

	\bibliographystyle{plain}
\bibliography{bibliography}

%\begin{thebibliography}{HD}
%\bibitem[W] {W}Weissler,F.,Two-point inequalities, theHermitesemigroup,andtheGauss-Weierstrasssemigroup.J.Funct. Anal.32(1979),no.1 --  dla nierownosci miedzy momentami dla p=1/2.
%\bibitem[D]{Dewar2017}  
%\end{thebibliography}

\end{document}